\documentclass[a4paper,10pt,article]{amsart}
\usepackage{amsmath,amscd,amssymb,amsthm,verbatim,indentfirst,dsfont,mathrsfs,ipa,graphicx,textcomp,enumerate}
\usepackage[all]{xy}
\usepackage{ipa,booktabs,longtable,xypic}
\usepackage[numbers]{natbib}
\usepackage{commutative-diagrams}
\usepackage[colorlinks=true,
            linkcolor=blue,
            anchorcolor=blue,
            citecolor=blue]{hyperref}

\newtheorem{theorem}{Theorem}[section]
\newtheorem{lemma}[theorem]{Lemma}
\newtheorem{corollary}[theorem]{Corollary}
\newtheorem{proposition}[theorem]{Proposition}
\newtheorem{conjecture}[theorem]{Conjecture}

\theoremstyle{definition}
\newtheorem{definition}[theorem]{Definition}
\newtheorem{example}[theorem]{Example}

\theoremstyle{remark}
\newtheorem{remark}[theorem]{Remark}

\newtheorem*{proof-thm-prod-coefficients}{Proof of Theorem \ref{prod-coefficients}}
\numberwithin{equation}{section}

\newcommand{\Z}{\mathbb{Z}}
\newcommand{\N}{\mathbb{N}}
\newcommand{\R}{\mathbb{R}}
\newcommand{\C}{\mathbb{C}}
\newcommand{\K}{\mathfrak{K}}

\newcommand{\XA}{\ell^2(Z_X)\otimes H \otimes H_A}
\newcommand{\YA}{\ell^2(Z_Y)\otimes H \otimes H_A}
\newcommand{\XH}{\ell^2(Z_X)\otimes H}
\newcommand{\YH}{\ell^2(Z_Y)\otimes H}
\newcommand{\U}{\mathcal{U}}
\DeclareMathOperator{\supp}{supp}
\DeclareMathOperator{\prop}{prop}

\title{The coarse Baum-Connes conjecture with filtered coefficients and product metric spaces}
\begin{document}

\author{Jianguo Zhang}
\address{School of Mathematics and Statistics, Shaanxi Normal University}
\email{jgzhang@snnu.edu.cn}

\thanks{This work was supported by NSFC (Nos. 12171156, 12271165, 12301154) and the Fundamental Research Funds for the Central Universities (No. 1301032574).}
\date{\today}

\begin{abstract}
Inspired by the quantitative $K$-theory, in this paper, we introduce the coarse Baum-Connes conjecture with filtered coefficients which generalizes the original conjecture. There are two advantages for the conjecture with filtered coefficients. Firstly, the routes toward the coarse Baum-Connes conjecture also work for the conjecture with filtered coefficients. Secondly, the class of metric spaces that satisfy the conjecture with filtered coefficients is closed under products and yet it is unknown for the original conjecture. As an application, we discover some new examples of product metric spaces for the coarse Baum-Connes conjecture.
\end{abstract}
\pagestyle{plain}
\maketitle


\section{Introduction}
The coarse Baum-Connes conjecture (cf. \cite{HigsonRoe-CBC}\cite{Yu-CBC}) can be dated back to the celebrated Atiyah-Singer index theorem which gives a formula by using the topological index to compute the Fredholm index of an elliptic differential operator on a closed manifold. For an elliptic differential operator on an open Riemannian manifold, Roe defined a higher index which lives in the $K$-theory of a $C^{\ast}$-algebra, now called the Roe algebra (cf. \cite{Roe1988}). Let us recall its definition for a locally finite metric space $X$. Let $\K(H)$ be the algebra of compact operators on a Hilbert space $H$. The \textit{Roe algebra} of $X$, denoted by $C^{\ast}(X)$, is defined to be the norm closure of the $\ast$-algebra consisting of all $(X\times X)$-matrices $T$ with coefficients in $\K(H)$ satisfying $T_{x,y}=0$ when $d(x, y)>R$ for some $R>0$.

The coarse Baum-Connes conjecture for $X$ asserts that the assembly map from the coarse $K$-homology of $X$ to the $K$-theory of the Roe algebra of $X$ is an isomorphism (see Conjecture \ref{CBC}). The left-hand side of the conjecture is topological and computable. Thus, the coarse Baum-Connes conjecture provides a topological algorithm for higher indexes of elliptic differential operators on Riemannian manifolds. Moreover, the conjecture has important applications to the Novikov conjecture on the homotopy invariance of higher signatures for closed manifolds, the Gromov-Lawson-Rosenberg conjecture concerning the existence of Riemannian metrics with positive scalar curvature on a manifold and Gromov's zero-in-the-spectrum conjecture about the spectrum of the Laplacian operator on a Riemannian manifold (cf. \cite{HigsonRoe-Book}\cite{Roe1993}\cite{WillettYu-Book}). 

There are three routes toward the coarse Baum-Connes conjecture. Roughly speaking, the first one is algebraic. In \cite{Yu1998}, Yu introduced the quantitative $K$-theory to verify that the conjecture holds for metric spaces with finite asymptotic dimension. The second one is geometric. In \cite{Yu2000}, Yu applied the Dirac-dual-Dirac method to prove that the conjecture holds for metric spaces which admit a coarse embedding into Hilbert space. The third one is topological. In \cite{HigsonRoe-CBC}, Higson and Roe employed the coarse homotopy and open cones to show that the conjecture is true for non-positively curved metric spaces. On the other hand, there exist some counterexamples to the coarse Baum-Connes conjecture, such as large spheres and certain expanders or asymptotic expanders (cf. \cite{HLS-2002}\cite{Khukhro-Li-Vigolo-Zhang}\cite{Willett-Yu-expander}\cite{Yu1998}).

In \cite{OyonoYu2015}, Oyono-Oyono and Yu generalized the quantitative $K$-theory to all filtered $C^{\ast}$-algebras as defined below.
\begin{definition}(\cite[Definition 1.1]{OyonoYu2015})\label{Def-filtration}
	Let $A$ be a $C^{\ast}$-algebra. A \textit{filtration} of $A$ is a family of closed linear subspaces $(A_r)_{r>0}$ such that
	\begin{itemize}
		\item $A_r\subseteq A_{r'}$, if $r\leq r'$;
		\item $A_r$ is closed by involution;
		\item $A_rA_{r'}\subseteq A_{r+r'}$;
		\item $\cup_{r>0}A_r$ is dense in $A$. 
	\end{itemize}
	If $A$ is unital, we also require the unit of $A$ belongs to $A_r$ for each $r>0$. A $C^{\ast}$-algebra equipped with a filtration is called a \textit{filtered $C^{\ast}$-algebra}.
\end{definition}
As a refinement of the $K$-theory, the quantitative $K$-theory with order $r$ of $A$ is defined by homotopy equivalent classes of almost-projections (or almost-unitaries) in $n\times n$-matrices with coefficients in $A_r$ for all $n\in \N$. 

Inspired by the quantitative $K$-theory, in this paper, we introduce the notion of the Roe algebra with filtered coefficients as below (see Definition \ref{Def-Roe-filcoeff} for the general case). 
\begin{definition}\label{Def-Roe-fil-intro}
	Let $X$ be a locally finite metric space and $A$ be a filtered $C^{\ast}$-algebra. The \textit{Roe algebra of $X$ with filtered coefficients in $A$}, denoted by $C^{\ast}_f(X, A)$, is defined to be the norm closure of the $\ast$-algebra consisting of all $(X\times X)$-matrices $T$ with coefficients in $\K(H)\otimes A_r$ satisfying $T_{x,y}=0$ when $d(x, y)>R$ for some $r, R>0$.
\end{definition} 
Then we introduce the coarse Baum-Connes conjecture with filtered coefficients to compute the $K$-theory of Roe algebras with filtered coefficients by using the $K$-theory of localization algebras with filtered coefficients (see Conjecture \ref{CBC}). When $A=\C$, the coarse Baum-Connes conjecture with filtered coefficients in $\C$ is exactly the original conjecture. Moreover, the above-mentioned three ways to the original conjecture also work for the coarse Baum-Connes conjecture with filtered coefficients (see Section \ref{Section-ex}).

In order to connect the coarse Baum-Connes conjecture with filtered coefficients for product metric spaces and the conjecture with filtered coefficients for their factors, we introduce a notion of localization algebras along one direction with filtered coefficients as a bridge (see Definition \ref{Def-localg-along-X}). Then by a Mayer-Vietoris argument, we obtain the following main result (see also Theorem \ref{main-corollary}).

\begin{theorem}\label{main-Thm}
	Let $X$, $Y$ be two proper metric spaces with bounded geometry and $A$ be a filtered $C^{\ast}$-algebra. If the following two conditions are satisfied:
	\begin{enumerate}
		\item $X$ satisfies the coarse Baum-Connes conjecture with filtered coefficients in the Roe algebra $C^{\ast}_{f}(Y, A)$;  
		\item $Y$ satisfies the coarse Baum-Connes conjecture with filtered coefficients in the uniform product $\prod^{u}_{\N}(A\otimes \K(H))$.
	\end{enumerate}
	Then $X\times Y$ satisfies the coarse Baum-Connes conjecture with filtered coefficients in $A$. 
	
	In particular, if $X$ and $Y$ satisfy the coarse Baum-Connes conjecture with filtered coefficients, then $X\times Y$ satisfies the coarse Baum-Connes conjecture with filtered coefficients.
\end{theorem}

Theorem \ref{main-Thm} is a coarse analogue of Oyono-Oyono's result in \cite[Corollary 7.12]{Oyono-BC-extensions} stating that the Baum-Connes conjecture with coefficients is closed under the direct product of groups. 

In \cite{Fukaya-Oguni-2015}, Fukaya and Oguni proved that the coarse Baum-Connes conjecture holds for product metric spaces of certain geometric groups by studying the coronae of products. In generally, if $X$ and $Y$ satisfy the coarse Baum-Connes conjecture, it is unknown whether the conjecture is true or not for $X\times Y$. Taking $A=\C$ in Theorem \ref{main-Thm}, we have the following corollary (see also Corollary \ref{corollay-CBC}), which not only recovers Fukaya and Oguni's result but also provides some new examples of products satisfying the conjecture (see Section \ref{Section-ex}). 

\begin{corollary}\label{Cor-intro}
	Let $X$ and $Y$ be two proper metric spaces with bounded geometry. If the followings are satisfied: 
	\begin{enumerate}
		\item $X$ satisfies the coarse Baum-Connes conjecture with filtered coefficients in the Roe algebra $C^{\ast}(Y)$;
		\item $Y$ satisfies the coarse Baum-Connes conjecture with coefficients in the direct product $\prod_{\N} \K(H)$.
	\end{enumerate}
	Then $X\times Y$ satisfies the coarse Baum-Connes conjecture. 
	
	In particular, if $X$ and $Y$ satisfy the coarse Baum-Connes conjecture with filtered coefficients, then $X\times Y$ satisfies the coarse Baum-Connes conjecture.
\end{corollary}

Corollary \ref{Cor-intro} is a general result for the coarse Baum-Connes conjecture of product metric spaces since we do not require any geometric conditions on each metric space.

Besides, the coarse Baum-Connes conjecture with filtered coefficients is also applied to the quantitative coarse Baum-Connes conjecture, which is a refinement of the original conjecture. We will explore this connection in a subsequent paper (cf. \cite{Zhang-quan-CBC}). 

The paper is organized as follows. In Section \ref{Sec-Pre}, we recall some ingredients from the $K$-theory of $C^{\ast}$-algebras. In Section \ref{Sec-CBC}, we introduce Roe algebras with filtered coefficients, localization algebras with filtered coefficients and the (consistent) coarse Baum-Connes conjecture with filtered coefficients. In Section \ref{main-sec}, we introduce the localization algebras along one direction for products to build a connection between the coarse Baum-Connes conjecture with filtered coefficients for products and the conjecture with filtered coefficients for their factors. And we prove Theorem \ref{main-Thm} in the end of this section. In Section \ref{Section-ex}, we apply our results to a number of product metric spaces.

\section{Preliminaries}\label{Sec-Pre}
In this section, we shall recall three important techniques for computing the $K$-theory of $C^{\ast}$-algebras. The main reference is Willett and Yu's book \cite{WillettYu-Book}.

\begin{lemma}\cite[Proposition 2.7.5]{WillettYu-Book}\label{isometryequivalent}
	Let $\alpha: C\rightarrow D$ be a $\ast$-homomorphism and $v$ be a partial isometry in the multiplier algebra of $D$ satisfying $\alpha(c)v^{\ast}v=\alpha(c)$ for all $c\in C$. Define $\beta: C\rightarrow D$ by $\beta(c)=v \alpha(c) v^{\ast}$. Then $\alpha$ and $\beta$ induce the same morphism on $K$-theory. 
\end{lemma}

\begin{definition}\cite[Definition 2.7.11]{WillettYu-Book}\label{Def-quasi-stable}
	A $C^{\ast}$-algebra $B$ is called to be \textit{quasi-stable} if for any positive integer $n$ there exists an isometry $v$ in the multiplier algebra of $M_n(B)$ such that $vv^{\ast}$ is the matrix unit $e_{11}$.
\end{definition}

One of the advantages of quasi-stable $C^{\ast}$-algebras is as follow.

\begin{lemma}\cite[Lemma 12.4.3]{WillettYu-Book}\label{continuous-flow-K}
	Let $B$ be a quasi-stable $C^{\ast}$-algebra and $C_{ub}([0,\infty), B)$ be the $C^{\ast}$-algebra consisting of all bounded, uniformly continuous functions from $[0,\infty)$ to $B$. Then the evaluation at zero map 
	$$e: C_{ub}([0,\infty), B) \rightarrow B$$
	induces an isomorphism on $K$-theory.
\end{lemma}

The following Mayer-Vietoris six-term exact sequence is a key technique to compute the $K$-theory of $C^{\ast}$-algebras by their ideals. Please refer to \cite[Proposition 2.7.15]{WillettYu-Book} or \cite{Higson-Roe-Yu} for the proof.

\begin{lemma}\label{Lem-MV-sequence}
	Let $I_1$ and $I_2$ be two closed ideals in a $C^{\ast}$-algebra $B$. Assume $I_1+I_2$ is dense in $B$. Then there exists a Mayer-Vietoris six-term exact sequence:
	$$\xymatrix{
		K_0(I_1\cap I_2) \ar[r] & 
		K_0(I_1)\oplus K_0(I_2) \ar[r] & 
		K_0(B)\ar[d] \\
		K_1(B)\ar[u] &
		K_1(I_1)\oplus K_1(I_2)\ar[l] &
		K_1(I_1\cap I_2),\ar[l] 
	}$$
	which is natural for triples $(I_1, I_2, B)$.
\end{lemma}

Now, Let us fix some notations in the rest of this paper. Let $(X,d)$ be a proper metric space, namely, every closed ball in $X$ is compact. Since proper metric spaces are separable, thus we can choose a countable dense subset $Z_X$ in $X$. Let $A$ be a filtered $C^{\ast}$-algebra acting on a Hilbert space $H_A$. Fixed $H$ as a separable Hilbert space. And let $\K(H')$ be the algebra of compact operators on a Hilbert space $H'$.

\section{The coarse Baum-Connes conjecture with filtered coefficients}\label{Sec-CBC}
In this section, we will introduce the coarse Baum-Connes conjecture with filtered coefficients which generalizes the original conjecture. 

\subsection{Roe algebras with filtered coefficients}
The Roe algebra was introduced by Roe in \cite{Roe1993} motivated by the index theory on open manifolds \cite{Roe1988}. In this subsection, we introduce a notion of the Roe algebra with filtered coefficients in a filtered $C^{\ast}$-algebra. 

\begin{definition}\label{Def-propagation}
	Let $X$ and $Y$ be two proper metric spaces. Let $Z_X$ and $Z_Y$ be two countable dense subsets of $X$ and $Y$, respectively. Let $\chi$ be the characteristic function. 
	\begin{enumerate}
		\item For a bounded operator $T: \XA\rightarrow \YA$. The \textit{support} of $T$, denoted by $\supp(T)$, is defined to be the set consisting of all $(x, y)\in X\times Y$ satisfying $(\chi_V\otimes I \otimes I) T (\chi_U\otimes I \otimes I) \neq 0$ for any open neighborhoods $U$ and $V$ of $x$ and $y$, respectively.
		\item For a bounded operator $T$ on $\XA$. The \textit{propagation} of $T$, denoted by $\prop(T)$, is defined to be $\sup\{d(x, y): (x,y)\in \supp(T)\}$.
	\end{enumerate}
\end{definition}

\begin{definition}\label{Def-Roe-filcoeff}
	The \textit{Roe algebra of $X$ with filtered coefficients in $A$}, denoted by $C^{\ast}_f(X, A)$, is defined to be the norm closure of the $\ast$-algebra consisting of all operators $T$ on $\XA$ satisfying that 
	\begin{itemize}
		\item $T$ has finite propagation;
		\item there exists $s>0$ such that $(\chi_K\otimes I \otimes I) T$ and $T (\chi_K\otimes I\otimes I)$ belong to $\K(\ell^2(Z_X) \otimes H) \otimes A_s$ for any compact subset $K\subseteq X$.
	\end{itemize}
\end{definition}

\begin{remark}
	If a $C^{\ast}$-algebra $A$ is endowed with the trivial filtration, namely, $A_r=A$ for each $r>0$. Then $C_f^{\ast}(X, A)$ is the \textit{Roe algebra of $X$ with coefficients in $A$}, also denoted by $C^{\ast}(X, A)$. In particular, the $C^{\ast}$-algebra $C_f^{\ast}(X, \C)$ or $C^{\ast}(X)$ for short, is so-called the \textit{Roe algebra} of $X$.
\end{remark}

Now, we discuss the functoriality on proper metric spaces for Roe algebras with filtered coefficients.

\begin{definition}\label{Def-coarse-map}
	Let $X$ and $Y$ be two proper metric spaces. A map $g: X\rightarrow Y$ is called a \textit{coarse map}, if 
	\begin{itemize}
		\item the inverse image of $g$ of any compact subset is pre-compact;
		\item for any $R>0$, there exists $S>0$ such that $\sup\{d(g(x),g(y)): d(x,y)\leq R\}\leq S$.  
	\end{itemize}
\end{definition}

For a coarse map $g$, define a function $w_g: \R^{+}\rightarrow \R^{+}$ by   $$w_g(R)=\sup\{d(g(x),g(y)):\:\: d(x,y)\leq R\}.$$

\begin{definition}\label{Def-coarse-equi}
	Two proper metric spaces $X$ and $Y$ are called to be \textit{coarsely equivalent}, if there exist two coarse maps $g: X\rightarrow Y$, $h:Y\rightarrow X$ and a constant $c>0$ such that 
	$$\max\{d(hg(x), x), d(gh(y), y)\}\leq c$$
	for any $x\in X$ and $y\in Y$.
\end{definition}

For a coarse map $g: X\rightarrow Y$ and $\delta>0$. Choosing a disjoint Borel cover $(U_i)_{i\in \N}$ of $Y$ with diameter less than $\delta$ and every $U_i$ has non-empty interior, then there exists a family of isometries $V_i: \ell^2(g^{-1}(U_i)\cap Z_X)\otimes H \rightarrow \ell^2(U_i\cap Z_Y)\otimes H$. Let 
$$V_g=\bigoplus_{i\in \N}V_i \otimes I_{H_A}: \XA \rightarrow \YA.$$
Then $V_g$ is an isometry and 
\begin{equation}\label{*}
	\supp(V_g)\subseteq \{(x, y)\in X\times Y: \: d(g(x),y)<\delta\}.
\end{equation}
Thus we have the following lemma.

\begin{lemma}\label{covering-isometry}
	For a coarse map $g: X\rightarrow Y$ and $\delta>0$, the above isometry $V_g$ induces a $\ast$-homomorphism
	$$ad_g: C_f^{\ast}(X, A)\rightarrow C_f^{\ast}(Y, A),$$
	defined by 
	$$ad_g(T)=V_g T V^{\ast}_g,$$
	that satisfies 
	$$\prop(ad_g(T))\leq w_g(\prop(T))+2\delta.$$
	Moreover, another $\ast$-homomorphism $ad'_g$ defined by another isometry $V'_g$ that satisfies the condition \eqref{*} induces the same homomorphism as $ad_g$ on $K$-theory.     
\end{lemma}

\begin{proof}
	Firstly, if $(y, y')\in \supp(V_g T V^{\ast}_g)$, then there exist $x, x'\in X$ such that $(y, x)\in \supp(V^{\ast}_g)$, $(x, x')\in \supp(T)$ and $(x', y')\in \supp(V_g)$. Thus 
	$$d(y, y')\leq d(y, g(x))+d(g(x), g(x'))+d(g(x'), y')\leq w_g(\prop(T))+2\delta,$$ 
	which implies $\prop(ad_g(T))\leq w_g(\prop(T))+2\delta$. 
	
	Secondly, for any compact subset $K$ of $Y$, let 
	$$K'=\{x\in X:\:\:\textrm{there exists $y\in K$ such that $(y, x)\in \supp(V^{\ast}_g)$}\}.$$
	Then we have $K'\subseteq g^{-1}(\overline{B_{\delta}(K)})$ since $d(g(x),y)<\delta$ for any $(y, x)\in \supp(V^{\ast}_g)$, where $\overline{B_{\delta}(K)}=\{y\in Y: d(y,K)\leq \delta\}$ is compact by the properness of $Y$. Thus $K'$ is a compact subset of $X$ by the properness of $g$. Moreover, since $T (\chi_{K'}\otimes I\otimes I) \in \K(\ell^2(Z_X)\otimes H) \otimes A_s$ and $V_g T V^{\ast}_g (\chi_K\otimes I\otimes I)=V_g T (\chi_{K'}\otimes I\otimes I) V^{\ast}_g (\chi_K\otimes I\otimes I)$. Thus, we have $V_g T V^{\ast}_g (\chi_K\otimes I\otimes I) \in \K(\ell^2(Z_X)\otimes H) \otimes A_s$. We can obtain the similar result for $(\chi_K\otimes I\otimes I) V_g T V^{\ast}_g$. In conclusion, $ad_g$ is well-defined.
	
	Finally, if there exists another isometry $V'_g: \XA \rightarrow \YA$ satisfying the condition \eqref{*}. Then by similar arguments as above, $V_gV^{\ast}_g$, $V_g'V'^{\ast}_g$, $V'_gV^{\ast}_g$ are all multipliers of $C_f^{\ast}(Y, A)$. And we have
	$$\begin{pmatrix}
		0 & 0\\
		0 & V'_g T V'^{\ast}_g 
	\end{pmatrix}
	=
	\begin{pmatrix}
		I-V_gV^{\ast}_g & V_gV'^{\ast}_g\\
		V'_gV^{\ast}_g  & I-V'_gV'^{\ast}_g 
	\end{pmatrix}
	\begin{pmatrix}
		V_g T V^{\ast}_g & 0\\
		0                  & 0 
	\end{pmatrix}
	\begin{pmatrix}
		I-V_gV^{\ast}_g & V_gV'^{\ast}_g\\
		V'_gV^{\ast}_g  & I-V'_gV'^{\ast}_g 
	\end{pmatrix},
	$$
	which implies that $ad'_g$ and $ad_g$ induce the same homomorphism on $K$-theory by Lemma \ref{isometryequivalent}. 
\end{proof}

If $g$ is the identity map on $X$, then we can choose $V_g$ to be a unitary. Thus we obtain the following corollary.

\begin{corollary}
	For different choices of the dense subset $Z_{X}$ of $X$, Roe algebras with filtered coefficients are non-canonically isomorphic and their $K$-theory are canonically isomorphic. 
\end{corollary}

\begin{corollary}\label{coarse-equi-isomor}
	If two proper metric spaces $X$ and $Y$ are coarsely equivalent, then 
	$$C_f^{\ast}(X, A)\cong C_f^{\ast}(Y, A)$$
	for any filtered $C^{\ast}$-algebras $A$.
\end{corollary}
\begin{proof}
	By the definition of coarse equivalence (see Definition \ref{Def-coarse-equi}), there exist two coarse maps $g: X \rightarrow Y$, $h: Y \rightarrow X$ and a constant $c>0$ such that 
	$$\max\{d(hg(x), x), d(gh(y), y)\}\leq c.$$ 
	Thus we can choose a disjoint Borel cover $(U_i)_{i\in \N}$ of $Y$ with every $U_i$ has non-empty interior such that $g^{-1}(U_i)\neq \emptyset$ and the diameter of $U_i$ is less than some positive number $\delta$ for any $i\in \N$. Let $Z_Y$ be a countable dense subset of $Y$ and $Z_X$ be a countable dense subset of $X$ with $Z_X \cap g^{-1}(U_i)\neq \emptyset$ for all $i\in \N$. Then for every $i\in \N$, there exists a family of unitaries $V_i: \ell^2(g^{-1}(U_i)\cap Z_X)\otimes H\rightarrow \ell^2(U_i\cap Z_Y)\otimes H$. Define
	$$V_g=\bigoplus_{i\in \N}V_i \otimes I_{H_A}: \XA \rightarrow \YA.$$
	Then $V_g$ is a unitary and we have
	$$\supp(V_g)\subseteq \{(x,y)\in X\times Y: \: d(g(x), y)<\delta\}.$$
	Thus by Lemma \ref{covering-isometry}, $ad_g$ is an isomorphism from $C_f^{\ast}(X, A)$ to $C_f^{\ast}(Y, A)$.
\end{proof}
In the end of this subsection, we discuss two properties of Roe algebras with filtered coefficients. Firstly, we show that they are filtered $C^{\ast}$-algebas equipped with the following filtrations.

\begin{lemma}\label{filtration-Roe-algebra}
	The Roe algebra $C_f^{\ast}(X, A)$ is a filtered $C^{\ast}$-algebra equipped with a filtration $C_f^{\ast}(X, A)_r$ consisting of all operators $T\in C_f^{\ast}(X, A)$ such that $\prop(T)\leq r$ and $(\chi_K\otimes I\otimes I)T, T(\chi_K\otimes I\otimes I)\in \K(\ell^2(Z_X)\otimes H)\otimes A_r$ for any compact subset $K\subseteq X$.
\end{lemma}

\begin{proof}
	For any two operators $T, S\in C_f^{\ast}(X, A)$, we have that $\supp(TS)$ is contained in the closure of the set 
	{\small$$\{(x,y)\in X\times X: \:\: \textrm{there exists $z\in X$ such that $(x,z)\in \supp(S)$, $(z,y)\in \supp(T)$}\},$$}
	which implies 
	$$\prop(TS)\leq \prop(T)+\prop(S).$$
	Besides, since $(A_r)_{r>0}$ is a filtration of $A$, thus $(C_f^{\ast}(X, A)_r)_{r>0}$ is a filtration of $C_f^{\ast}(X, A)$.
\end{proof}
Secondly, we show that Roe algebras with filtered coefficients are quasi-stable.

\begin{lemma}\label{Roe-quasi-stable}
	Let $X$ be a proper metric space and $A$ be a filtered $C^{\ast}$-algebra acting on a Hilbert space $H_{A}$. Then the Roe algebra $C_{f}^{\ast}(X, A)$ is quasi-stable.
\end{lemma}
\begin{proof}
	Let $Z_{X}$ be a countable dense subset of $X$ and $H$ be a separable Hilbert space. For any positive integer $n$, define an isometry $V: H^{n} \rightarrow H \oplus H^{n-1}$ with the first coordinate $H$ as its image. Then $I \otimes V \otimes I: \ell^2(Z_{X}) \otimes H^{n} \otimes H_{A} \rightarrow \ell^2(Z_{X}) \otimes H^{n} \otimes H_{A}$ is an isometry in the multiplier algebra of $M_n(C_{f}^{\ast}(X, A))$ with the propagation zero. Moreover, we have that $(I \otimes V \otimes I)(I \otimes V \otimes I)^{\ast}=I \otimes VV^{\ast} \otimes I=e_{11}$. Thus $C_{f}^{\ast}(X, A)$ is quasi-stable.
\end{proof}

\subsection{Localization algebras with filtered coefficients}
Yu defined a notion of the localization algebra whose $K$-theory provides a model for $K$-homology in \cite{Yu-Localizationalg}. In this subsection, we introduce the localization algebra with filtered coefficients.   

\begin{definition}\label{Def-Localg-filcoeff}
	The \textit{localization algebra of $X$ with filtered coefficients in $A$}, denoted by $C^{\ast}_{L,f}(X, A)$, is defined to be the norm closure of the $\ast$-algebra consisting of all bounded and uniformly continuous functions $u:[0,\infty)\rightarrow C^{\ast}_{f}(X, A)$ satisfying
	$$\lim_{t\rightarrow \infty}\prop(u(t))=0.$$
\end{definition}

\begin{remark}
	If a $C^{\ast}$-algebra $A$ is endowed with the trivial filtration, then $C^{\ast}_{L,f}(X, A)$ is the \textit{localization algebra of $X$ with coefficients in $A$}, also denoted by $C^{\ast}_{L}(X, A)$. In particular, the $C^{\ast}$-algebra $C^{\ast}_{L,f}(X, \C)$ or $C^{\ast}_L(X)$ for short, is so-called the \textit{localization algebra of $X$} (cf. \cite{Yu-Localizationalg}).
\end{remark}

Let $g: X\rightarrow Y$ be a uniformly continuous coarse map and $\{\delta_k\}_{k\in \N}$ be a
decreasing sequence of positive numbers with limit zero. Choosing a sequence of Borel covers $\{\U_k\}_{k\in \N}$ of $Y$ such that 
\begin{itemize}
	\item every cover $\U_k$ is disjoint, i.e. $U_{k,i}\cap U_{k,j}=\emptyset$ for different $i, j$;
	\item every $U_{k,i}\in \U_k$ has non-empty interior;
	\item the diameter of $\U_k$ is less than $\delta_k$ for any $k\in \N$;
	\item $\U_{k+1}$ refines $\U_k$, i.e. every set in $\U_{k+1}$ is contained in some set in $\U_k$ for all $k\in \N$.
\end{itemize}
Then $\ell^2(U_{k,i}\cap Z_Y)\otimes H$ is an infinite dimensional separable Hilbert space for every set $U_{k,i}\in \U_{k}$. Thus we can define an isometry $V_{k,i}: \ell^2(g^{-1}(U_{k,i})\cap Z_X)\otimes H \rightarrow \ell^2(U_{k,i}\cap Z_Y)\otimes H$ with infinite dimensional cokernel. By the fourth condition as above, there exists $\{U_{k+1,ij}\}_{1\leq j \leq n}\subseteq \U_{k+1}$ such that $\cup^n_{j=1}U_{k+1,ij}=U_{k,i}$. Then we can choose a family of isometries $V_{k+1,ij}: \ell^2(g^{-1}(U_{k+1,ij})\cap Z_X)\otimes H \rightarrow \ell^2(U_{k+1,ij}\cap Z_Y)\otimes H$ with infinite dimensional cokernel. Let $V_{k+1,i}=\bigoplus^n_{j=1}V_{k+1,ij}: \ell^2(g^{-1}(U_{k,i})\cap Z_X)\otimes H \rightarrow \ell^2(U_{k,i}\cap Z_Y)\otimes H$ which is an isometry with infinite dimensional cokernel. Let $W_{k, i}$ be a unitary from the cokernel of $V_{k,i}$ to the cokernel of $V_{k+1,i}$ and let $W'_{k,i}=V_{k+1,i}V^{\ast}_{k,i}+0\oplus W_{k,i}$, which is a unitary on $\ell^2(U_{k,i}\cap Z_Y)$. Let $h: S^1\rightarrow [0,1)$ be the Borel inverse function of the continuous function $t'\mapsto e^{2\pi i t'}$. Define $\gamma_{k,i}$ to be a path of isometries by $\gamma_{k,i}(t)=e^{2\pi ith(W'_{k, i})}V_{k,i}$. Then $\gamma_{k,i}$ is a $2\pi$-Lipschitz path connecting $V_{k, i}$ with $V_{k+1,i}$ and $\supp(\gamma_{k,i}(t))\subseteq g^{-1}(U_{k,i}) \times U_{k,i}$. Define 
$$\gamma_k(t)=\bigoplus_i \gamma_{k,i}(t): \XH \rightarrow \YH$$
for $t\in [0,1]$. Then for $s\in [0, \infty)$, define
$$V_g(s)=\gamma_k(s-k) \otimes I_{H_A}: \ell^2(Z_X)\otimes H \otimes H_A \rightarrow \ell^2(Z_Y)\otimes H \otimes H_A,\:\:\textrm{if $s\in [k,k+1]$}.$$
Then $V_g$ is a bounded and uniformly continuous function satisfying
\begin{equation}\label{**}
	\supp(V_g(s))\subseteq \{(x,y)\in X\times Y: d(g(x), y)<\delta_k\},\:\:\textrm{if $s\in[k,k+1]$}.
\end{equation}
Thus by Lemma \ref{covering-isometry}, we have the following lemma.

\begin{lemma}\label{continuous-covering-isometry}
	Let $g: X\rightarrow Y$ be a uniformly continuous coarse map and $\{\delta_k\}_{k\in \N}$ be a decreasing sequence of positive numbers with limit zero. Then $V_g$ as defined above induces a $\ast$-homomorphism 
	$$Ad_g: C^{\ast}_{L, f}(X, A) \rightarrow C^{\ast}_{L, f}(Y, A)$$
	defined by 
	$$Ad_g(u)(t)=V_g(t)u(t)V^{\ast}_g(t),$$
	that satisfies 
	$$\prop(Ad_g(u)(t))\leq w_g(\prop(u(t)))+2\delta_k\:\:\textrm{for $t\in [k,k+1]$}.$$
	Moreover, for another bounded and uniformly continuous function $V'_g$ that maps $[0,\infty)$ to the set of all isometries on $\ell^2(Z_X)\otimes H \otimes H_A$ and satisfies the condition \eqref{**}, then $Ad'_g$ defined by $V'_g$ induces the same homomorphism as $Ad_g$ on $K$-theory.
\end{lemma} 

If replacing the Hilbert space $H$ by $H^{\infty}=\oplus_{i\geq 0} H$ in the definition of the localization algebra with filtered coefficients, then we obtain a $C^{\ast}$-algebra, denoted by $C^{\ast}_{L, f}(X, A, H^{\infty})$. Thus by Lemma \ref{isometryequivalent}, we have the following lemma.

\begin{lemma}\label{Lem-dep-Hilspace}
	Define $V: H\rightarrow H^{\infty}$ by $V(\xi)=(\xi, 0,\cdots,0)$, then the map from $C^{\ast}_{L,f}(X, A)$ to $C^{\ast}_{L, f}(X, A, H^{\infty})$ defined by $u(t)\mapsto (I\otimes V \otimes I)u(t)(I\otimes V^{\ast} \otimes I)$ induces an isomorphism on the $K$-theory.	 
\end{lemma} 

Next, we will show some homological properties of the $K$-theory of the localization algebras with filtered coefficients.

\begin{definition}\cite[Definition 3.5]{Yu-Localizationalg}\label{Def-strLip}
	Let $X$, $Y$ be two proper metric spaces and $g$, $h$ be two Lipschitz maps from $X$ to $Y$. A continuous homotopy $F(t, x)$ $(t\in [0,1])$ between $g$ and $h$ is said to be \textit{strongly Lipschitz} if
	\begin{enumerate}
		\item $F(t, x)$ is a coarse map from $X$ to $Y$ for each $t\in [0,1]$;
		\item $d(F(t, x), F(t, y))\leq Cd(x, y)$ for all $x, y\in X$ and $t\in [0,1]$, where $C$ is a constant (called Lipschitz constant of $F$);
		\item for any $\varepsilon>0$, there exists $\delta>0$ such that $d(F(t_1, x), F(t_2, x))<\varepsilon$ for all $x\in X$ if $\|t_1-t_2\|<\delta$;
		\item $F(0,x)=g(x)$ and $F(1, x)=h(x)$ for all $x\in X$.
	\end{enumerate} 
\end{definition}

\begin{definition}\cite[Definition 3.6]{Yu-Localizationalg}\label{Def-strLip-equi}
	$X$ is said to be \textit{strongly Lipschitz homotopy equivalent} to $Y$ if there exist two Lipschitz maps $g: X\rightarrow Y$ and $h:Y\rightarrow X$ such that $hg$ and $gh$ are strongly Lipschitz homotopic to $id_{X}$ and $id_{Y}$, respectively. 
\end{definition}

The $K$-theory of localization algebras with filtered coefficients is computable because of the following two lemmas, the proofs are coming from Yu's paper \cite{Yu-Localizationalg}.

\begin{lemma}\label{Lip-homotopy-invaraint}
	If $X$ is strongly Lipschitz homotopy equivalent to $Y$, then $K_{\ast}(C^{\ast}_{L, f}(X, A))$ is naturally isomorphic to $K_{\ast}(C^{\ast}_{L, f}(Y, A))$.
\end{lemma}
\begin{proof}
	By the assumption, there exists two Lipschitz maps $g: X\rightarrow Y$ and $h:Y\rightarrow X$ such that $hg$ and $gh$ are strongly Lipschitz homotopic to $id_X$ and $id_Y$, respectively. It is sufficient to prove that $Ad_{hg}$ induces the identity map on $K_{\ast}(C^{\ast}_{L,f}(X, A))$ and similar to $Ad_{gh}$. Let $F(t,x)$ be a strongly Lipschitz homotopy with the Lipschitz constant $C$ such that $F(0,\cdot)=hg$ and $F(1,\cdot)=id_X$. Then we can choose a sequence of non-negative numbers $\{t_{i,j}\}_{i,j\geq 0}$ in $[0,1]$ such that
	\begin{itemize}
		\item $t_{0, j}=0$ and $t_{i, j}\leq t_{i+1,j}$;
		\item for each $j$, there exists $N_j\geq 0$ such that $t_{i,j}=1$ for all $i\geq N_j$;
		\item $d(F(t_{i,j}, x), F(t_{i+1,j}, x))\leq 1/(j+1)$ and $d(F(t_{i,j}, x), F(t_{i,j+1}, x))\leq 1/(j+1)$ for all $x\in X$.
	\end{itemize}
	By the construction before Lemma \ref{continuous-covering-isometry}, there exists a sequence of multipliers $V_{F(t_{i,j},\cdot)}$ of $C^{\ast}_{L,f}(X,A)$ such that
	$$\supp(V_{F(t_{i,j},\cdot)}(t))\subseteq \{(x,x')\in X\times X: d(F(t_{i,j},x), x')<1/(1+i+j)\}$$
	for any $i,j\geq 0$. 
	Define
	$$V_i(t)=R(t-j) (V_{F(t_{i,j},\cdot)}(t)\oplus V_{F(t_{i,j+1},\cdot)}(t)) R^{\ast}(t-j),\: t\in [j,j+1],$$
	where 
	$$R(t)=
	\begin{pmatrix}
		\cos(\pi t/2) & \sin(\pi t/2)\\
		-\sin(\pi t/2) & \cos(\pi t/2)
	\end{pmatrix}.$$
	Now we prove the lemma for $\ast=1$ and by a suspension argument, we can similarly prove it for $\ast=0$. For an element $[u]\in K_{1}(C^{\ast}_{L, f}(X, A))$, we assume $u\in C^{\ast}_{L, f}(X, A)^{+}$. Let $W=u\oplus I$. Define the following three unitaries:
	\begin{equation*}
		\begin{split}
			&a_1=\bigoplus_{i\geq 0} (V_i(t)((u(t)-I)\oplus 0)V^{\ast}_i(t)+I)\cdot W^{\ast};\\
			&a_2=\bigoplus_{i\geq 0} (V_{i+1}(t)((u(t)-I)\oplus 0)V^{\ast}_{i+1}(t)+I)\cdot W^{\ast};\\
			&a_3=I\bigoplus_{i\geq 1} (V_i(t)((u(t)-I)\oplus 0)V^{\ast}_i(t)+I)\cdot W^{\ast}.
		\end{split}
	\end{equation*}
	Since $\prop(a_i)\leq (C+1) \prop(u(t))+1/(j+1)$ for $t\in [j,j+1]$ and $V_i(t)$ acts on $H_A$ as the identity operator, thus $a_i\in M_2(C^{\ast}_{L, f}(X, A, H^{\infty}))^{+}$. Moreover, since $d(F(t_{i,j}, x), F(t_{i+1,j}, x))\leq 1/(j+1)$, thus $[a_1]=[a_2]$ in $K_{1}(C^{\ast}_{L, f}(X, A))$ (the proof is similar to the last section of the proof of Lemma \ref{covering-isometry}). Let 
	$$V: H^{\infty}\rightarrow H^{\infty}, \: (h_1, h_2, \cdots)\mapsto (0, h_1, \cdots),$$ 
	then $I_{\ell^2(Z_X)}\otimes V \otimes I_{H_A}$ is an isometric multiplier of $C^{\ast}_{L,f}(X, A, H^{\infty})$. And 
	$$a_3=(I_{\ell^2(Z_X)}\otimes V \otimes I_{H_A})(a_2-I)(I_{\ell^2(Z_X)}\otimes V^{\ast} \otimes I_{H_A})+I,$$ 
	thus $[a_2]=[a_3]$ in $K_{1}(C^{\ast}_{L, f}(X, A))$ by Lemma \ref{isometryequivalent}. Then, we have that $$[a_1a^{-1}_3]=[(V_i(t)((u(t)-I)\oplus 0)V^{\ast}_i(t)+I)\cdot W^{\ast}\bigoplus_{i\geq 1}I]=[Ad_{hg}(u)W^{\ast}\bigoplus_{i\geq 1}I]$$
	which is equal to $[I]$. Thus $Ad_{hg, \ast}=id$ on $K_{1}(C^{\ast}_{L, f}(X, A))$ by Lemma \ref{Lem-dep-Hilspace}.
\end{proof}

\begin{definition}
	A cover $\{X_1, X_2\}$ of $X$ is called to be \textit{uniformly excisive}, if for any $\varepsilon_t\geq 0$ with $\lim_{t\rightarrow \infty} \varepsilon_t=0$, there exist $\delta_t\geq 0$ with $\lim_{t\rightarrow \infty} \delta_t=0$ such that
	$$B_{\varepsilon_t}(X_1)\cap B_{\varepsilon_t}(X_2)\subseteq B_{\delta_t}(X_1\cap X_2)$$
	for any $t\geq 0$, where $B_{\varepsilon_t}(X_i)$ is the $\varepsilon_t$-neighborhood of $X_i$. 
\end{definition}

\begin{lemma}\label{MV-fil-localg}
	Let $X=X_1\cup X_2$ be a uniformly excisive cover. If there exists $\varepsilon_0>0$ such that the $\varepsilon_0$-neighborhoods of $X_1, X_2$ and $X_1\cap X_2$ are strongly Lipschitz homotopy equivalent to $X_1, X_2$ and $X_1\cap X_2$, respectively. Then there is a Mayer-Vietoris six-term exact sequence
	$$\small\xymatrix{
		K_0(C^{\ast}_{L, f}(Y, A)) \ar[r] & 
		K_0(C^{\ast}_{L, f}(X_1, A))\oplus K_0(C^{\ast}_{L, f}(X_2, A)) \ar[r] & 
		K_0(C^{\ast}_{L, f}(X, A))\ar[d] \\
		K_1(C^{\ast}_{L, f}(X, A))\ar[u] &
		K_1(C^{\ast}_{L, f}(X_1, A))\oplus K_1(C^{\ast}_{L, f}(X_2, A))\ar[l] &
		K_1(C^{\ast}_{L, f}(Y, A)),\ar[l] 
	}$$
	where $Y=X_1\cap X_2$.
\end{lemma}
\begin{proof}
	For $i=1,2$, let $C^{\ast}_{L,f}(X, A; X_i)$ be a closed subalgebra of $C^{\ast}_{L,f}(X, A)$ generated by all elements $u$ satisfying that there exists $c_t\geq 0$ with $\lim_{t\rightarrow \infty}c_t=0$ such that 
	$$\supp(u(t))\subseteq \{(x,x')\in X\times X: d((x,x'), X_i\times X_i)\leq c_t\}$$
	for all $t\in [0,\infty)$. Then there exists a natural inclusion map 
	$$\tau: C^{\ast}_{L,f}(X_i, A)\rightarrow C^{\ast}_{L,f}(X, A; X_i).$$
	Now we prove that $\tau$ induces an isomorphism on the $K$-theory. By the assumption, there exists $\varepsilon_0>0$ such that the $\varepsilon_0$-neighborhood $B_{\varepsilon_0}(X_i)$ of $X_i$ is strongly Lipschitz homotopy equivalent to $X_i$. For any element $u\in C^{\ast}_{L,f}(X, A; X_i)$, there exists $t_0>0$ such that 
	$$\supp(u(t+t_0))\subseteq B_{\varepsilon_0}(X_i)\times B_{\varepsilon_0}(X_i)$$
	for all $t\geq 0$. Let $v(t)=u(t+t_0)$, then $u$ and $v$ are homotopic equivalent in $C^{\ast}_{L,f}(X, A; X_i)$ since $u$ is uniformly continuous. Thus the inclusion map from $C^{\ast}_{L,f}(B_{\varepsilon_0}(X_i), A)$ to $C^{\ast}_{L,f}(X, A; X_i)$ induces an isomorphism on the $K$-theory. Then $\tau_{\ast}: K_{\ast}(C^{\ast}_{L,f}(X_i, A))\rightarrow K_{\ast}(C^{\ast}_{L,f}(X, A; X_i))$ is an isomorphism by Lemma \ref{Lip-homotopy-invaraint}.
	
	On the other hand, $C^{\ast}_{L, f}(X, A; X_i)$ is an ideal of $C^{\ast}_{L, f}(X, A)$ for $i=1,2$. Moreover, we have $C^{\ast}_{L, f}(X, A; X_1)+C^{\ast}_{L, f}(X, A; X_2)=C^{\ast}_{L, f}(X, A)$ and $C^{\ast}_{L, f}(X, A; X_1) \cap C^{\ast}_{L, f}(X, A; X_2)=C^{\ast}_{L, f}(X, A; X_1\cap X_2)$. Thus we completed the proof by Lemma \ref{Lem-MV-sequence}.  
\end{proof}

\subsection{The coarse Baum-Connes conjecture with filtered coefficients}
Before building the coarse Baum-Connes conjecture with filtered coefficients, we need introduce a class of coarsening spaces for metric spaces by Rips
complexes.

Let $(N, d_N)$ be a locally finite metric space, namely, any ball in $N$ just contains a finite number of elements. For any $k>0$, the \textit{Rips complex} of $N$ at scale $k$, denoted by $P_k(N)$, is a simplicial complex whose set of vertices is $N$ and a subset $\{z_0, \cdots, z_n\}\subseteq N$ spans an $n$-simplex in $P_k(N)$ if and only if $d_N(z_i, z_j)\leq k$ for any $i,j=0, \cdots, n$. 

Firstly, define a metric $d_{S_k}$ on $P_{k}(N)$ to be the path metric whose restriction to each simplex is the standard spherical metric on the unit sphere by mapping $\sum_{i=0}^{n} t_{i}z_{i}$ to 
$$\left(\frac{t_0}{\sqrt{\sum_{i=0}^{n}t^{2}_{i}}}, \cdots, \frac{t_n}{\sqrt{\sum_{i=0}^{n}t^{2}_{i}}}\right),$$
namely, 
$$d_{S_k}\left(\sum_{i=0}^{n} t_{i}z_{i}, \sum_{i=0}^{n} t'_{i}z_{i}\right)=\frac{2}{\pi} \arccos\left(\frac{\sum_{i=0}^{n} t_i t'_i}{\sqrt{(\sum_{i=0}^{n}t^{2}_{i})(\sum_{i=0}^{n}t'^{2}_{i})}}\right),$$
where $t_i, t'_i\in [0,1]$ and $\sum_{i=0}^n t_i=\sum_{i=0}^n t'_i=1$. The metric $d_{S_k}$ between different connected components is defined to be infinity. Unfortunately, the natural include map from $(N, d_N)$ to $(P_k(N), d_{S_k})$ is generally not a coarsely equivalent map. For this reason, we need introduce the following modified metric on $P_k(N)$ coming from \cite[Definition 7.2.8]{WillettYu-Book}. 

Define a metric $d_{P_k}$ on $P_k(N)$ to be
$$d_{P_k}(x, y)=\inf\left\{\sum_{i=0}^{n} d_{S_k}(x_i, y_i)+\sum_{i=0}^{n-1} d_N(y_i, x_{i+1}) \right\},$$
for all $x, y\in P_k(N)$, here the infimum is taken over all sequences of the form $x=x_0, y_0, x_1, y_1,\cdots, x_n, y_n=y$, where $x_1,\cdots, x_n, y_0, \cdots, y_{n-1} \in N$.

Then by Lemma 7.2.9 and Proposition 7.2.11 in \cite{WillettYu-Book}, we have the following lemma.

\begin{lemma}\label{coarse-equi-Rips}
	For any $k\geq 0$, the metric spaces $(N, d_N)$ and $(P_k(N), d_{P_k})$ are coarsely equivalent. More precisely, let $i_k: N\rightarrow P_k(N)$ be the inclusion map and define $j_k: P_k(N)\rightarrow N$ by mapping $\sum_{i=0}^{n}t_iz_i$ to $z_i$ for some $t_i\neq 0$, then $i_k$ and $j_k$ are coarse maps and 
	$$\max\{d_{N}(j_{k}i_{k}(z),z), d_{P_k}(i_{k}j_{k}(x), x)\}\leq 1,$$
	for any $z\in N$, $x\in P_k(N)$.
\end{lemma}

For $k_1\leq k_2$, the natural inclusion map 
$$i_{k_1k_2}: P_{k_1}(N)\rightarrow P_{k_2}(N)$$
is a contractible coarse map. Then they induce the following homomorphisms by Lemma \ref{covering-isometry} and Lemma \ref{continuous-covering-isometry}:
$$ad_{i_{k_1k_2},\ast}: K_{\ast}(C_{f}^{\ast}(P_{k_1}(N), A))\rightarrow K_{\ast}(C_{f}^{\ast}(P_{k_2}(N), A)),$$
$$Ad_{i_{k_1k_2},\ast}: K_{\ast}(C_{L,f}^{\ast}(P_{k_1}(N), A))\rightarrow K_{\ast}(C_{L,f}^{\ast}(P_{k_2}(N), A)).$$
Thus we obtain two inductive systems of abelian groups: $$\{K_{\ast}(C_{f}^{\ast}(P_{k}(N), A)), ad_{kk', \ast}\} \:\:\text{and}\:\: \{K_{\ast}(C_{L, f}^{\ast}(P_{k}(N), A)), Ad_{kk', \ast}\}.$$

Consider the evaluation at zero map:
$$e: C_{L,f}^{\ast}(P_{k}(N), A)\rightarrow C_{f}^{\ast}(P_{k}(N), A),\: u\mapsto u(0),$$
which induces a homomorphism on $K$-theory level
$$e_{\ast}: K_{\ast}(C_{L,f}^{\ast}(P_{k}(N), A))\rightarrow K_{\ast}(C_{f}^{\ast}(P_{k}(N), A)),$$
that satisfies $ad_{i_{k_1k_2}, \ast}\circ e_{\ast}=e_{\ast} \circ Ad_{i_{k_1k_2}, \ast}$.

\begin{definition}\label{Def-net}
	Let $X$ be a metric space, a \textit{net} in $X$ is a subset $N_{X}$ satisfying that there exists $C\geq 1$ such that
	\begin{enumerate}
		\item $d(z, z')\geq C$ for any $z,z'\in N_{X}$;
		\item for any $x\in X$, there exists $z\in N_{X}$ such that $d(z, x)\leq C$.
	\end{enumerate}
\end{definition} 

\begin{remark}     
	By Zorn's lemma, any proper metric space admits a locally finite net. Moreover, any net $N_X$ is coarsely equivalent to $X$.
\end{remark}

Now, we are ready to introduce the \textit{coarse Baum-Connes conjecture with filtered coefficients} in a filtered $C^{\ast}$-algebra $A$ as follow:
\begin{conjecture}\label{CBC}
	Let $X$ be a proper metric space and $N_{X}$ be a locally finite net in $X$, then the following homomorphism
	$$e_{\ast}: \lim_{k\rightarrow \infty}K_{\ast}(C_{L,f}^{\ast}(P_{k}(N_X), A)) \rightarrow \lim_{k\rightarrow \infty}K_{\ast}(C_{f}^{\ast}(P_{k}(N_X), A))$$
	is an isomorphism between abelian groups. 
\end{conjecture}

\begin{remark} We have the following comments for the above conjecture.
	\begin{enumerate}
		\item The conjecture is independent of the choice of locally finite nets $N_{X}$ of $X$. More generally, the coarse Baum-Connes conjecture with filtered coefficients is invariant under coarse equivalences by Corollary \ref{coarse-equi-isomor}. 
		\item By Corollary \ref{coarse-equi-isomor} and Lemma \ref{coarse-equi-Rips}, we have that 
		$$\lim_{k\rightarrow \infty}K_{\ast}(C_{f}^{\ast}(P_{k}(N_X), A))\cong K_{\ast}(C_{f}^{\ast}(X, A)).$$
		Thus the coarse Baum-Connes conjecture with filtered coefficients provides a formula to compute the $K$-theory of Roe algebras with filtered coefficients.
		\item If $A$ is equipped with the trivial filtration, namely, $A_r=A$ for each $r>0$. Then the conjecture with filtered coefficients in $A$ is also called the \textit{coarse Baum-Connes conjecture with coefficients} in $A$. In particular, the coarse Baum-Connes conjecture with filtered coefficients in $\C$ is so-called the \textit{coarse Baum-Connes conjecture} (cf. \cite{HigsonRoe-CBC}, \cite{Yu-CBC}).
	\end{enumerate}
\end{remark}

At the last of this section, we consider a consistent version of the coarse Baum-Connes conjecture with filtered coefficients. 

Let $(X_i, d_i)_{i\in \N}$ be a sequence of metric spaces. Define the \textit{separated coarse union} of $(X_i)_{i\in \N}$ to be a metric space $(\sqcup_{\N} X_i, d)$ by
\begin{equation}\label{productmetric}
	d(x,x')=\left \{ 
	\begin{array}{lcl}
		d_i(x,x'), & \mbox{for} & x,x'\in X_i\\
		\infty,    & \mbox{for} & x\in X_i, x'\in X_{i'}, i\neq i'. 
	\end{array}
	\right.
\end{equation}
Then the Roe algebra $C_{f}^{\ast}(\sqcup_{\N} X_i, A)$ is the norm closure in $\prod_{i\in \N} C_{f}^{\ast}(X_i, A)$ of the set
$$\bigcup_{r>0}\left(\prod_{i\in \N} C_{f}^{\ast}(X_i, A)_{r}\right),$$
where $(C_{f}^{\ast}(X_i, A)_{r})_{r>0}$ is a filtration of Roe algebra $C_{f}^{\ast}(X_i, A)$ defined in Lemma \ref{filtration-Roe-algebra}. Thus, the localization algebra $C_{L, f}^{\ast}(\sqcup_{\N} X_i, A)$ is the norm closure of the $\ast$-algebra consisting of all bounded and uniformly continuous functions $u:[0,\infty) \rightarrow C_{f}^{\ast}(\sqcup_{\N} X_i, A)$ satisfying 
$$\lim_{t\rightarrow \infty} \sup_{i\in \N}(\prop(u_i(t)))=0,$$
where $u_i(t)$ is the $i$th coordinate of $u(t)$ in $C_{f}^{\ast}(X_i, A)$.

We call the following conjecture as the \textit{consistent coarse Baum-Connes conjecture with filtered coefficients} in $A$.

\begin{conjecture}\label{Def-uniform-CBC}
	Let $X$ be a proper metric space and $N_X$ be a locally finite net in $X$, then the following homomorphism
	$$e_{\ast}: \lim_{k\rightarrow \infty} K_{\ast}(C_{L,f}^{\ast}(\sqcup_{\N} \left(P_{k}(N_X)\right), A)) \rightarrow \lim_{k\rightarrow \infty} K_{\ast}(C_{f}^{\ast}(\sqcup_{\N} \left(P_{k}(N_X)\right), A))$$
	is an isomorphism between abelian groups, namely, the coarse Baum-Connes conjecture with filtered coefficients in $A$ holds for $\sqcup_{\N} X$.
\end{conjecture}

\begin{remark}
	Actually, the consistent coarse Baum-Connes conjecture with coefficients in $A$ is equivalent to the quantitative coarse Baum-Connes conjecture with coefficients in $A$. Please refer to \cite{Zhang-quan-CBC} for more details.
\end{remark}

\section{Product metric spaces}\label{main-sec}
In this section, we will prove that the coarse Baum-Connes conjecture with filtered coefficients for the product metric space of two metric spaces providing each of them satisfies the conjecture with filtered coefficients. 

Let $(X, d_X)$ and $(Y, d_Y)$ be two proper metric spaces, then the product space $X\times Y$ equipped with the following metric 
$$d((x,y),(x',y'))=\max\{d_X(x,x'),d_Y(y,y')\}\:\:\textrm{for}\:\: x,x'\in X \:\:\textrm{and}\:\:y,y'\in Y,$$
is also a proper metric space. 
\begin{remark}
	If $d'$ is another metric defined on $X\times Y$ that satisfies 
	\begin{itemize}
		\item $d'((x,y),(x',y))=d_X(x,x')$ and $d'((x,y),(x,y'))=d_Y(y,y')$ for any $x, x'\in X$ and $y, y'\in Y$;
		\item $d'((x, y), (x', y'))\geq \max\{d_X(x,x'), d_Y(y, y')\}$ for any $x, x'\in X$ and $y, y'\in Y$. 
	\end{itemize}
	Then $d'$ is bi-Lipschitz equivalent to $d$ as defined above, namely, 
	$$d((x, y),(x', y'))\leq d'((x, y),(x', y')) \leq 2d((x, y),(x', y')).$$
	For examples, we can define $d'((x, y), (x', y'))=(d_X(x,x')^p+d_Y(y,y')^p)^{1/p}$ for any $p\geq 1$.
\end{remark}

\subsection{The coarse Baum-Connes conjecture along one direction with filtered coefficients}
In this subsection, we build a bridge called \textit{the localization algebra along one direction with filtered coefficients} to connect the coarse Baum-Connes conjecture with filtered coefficients for products and the conjecture with filtered coefficients for their factors.

\begin{definition}\label{Def-prop-along-X}
	Let $(X, d_{X})$ and $(Y, d_{Y})$ be two proper metric spaces with two countable dense subsets $Z_X$ and $Z_Y$, respectively. Let $A$ be a filtered $C^{\ast}$-algebra acting on a Hilbert space $H_A$. Let $T$ be an operator acting on $\ell^2(Z_X)\otimes \ell^2(Z_Y)\otimes H\otimes H_A$. The \textit{propagation along $X$} of $T$, denoted by $\prop_X(T)$, is defined to be 
	$$\sup\{d_{X}(x,x'): \:\text{there exist $y, y'\in Y$ such that $\left(\left( x, y \right)\left(x', y' \right)\right)\in \supp(T)$}\}.$$
\end{definition}

\begin{definition}\label{Def-localg-along-X}
	Let $X$ and $Y$ be as above, the \textit{localization algebra along $X$ with filtered coefficients} in $A$ of $X\times Y$, denoted by $C^{\ast}_{L,X, f}(X\times Y, A)$, is defined to be the norm closure of the $\ast$-algebra consisting of all bounded and uniformly continuous functions $u:[0, \infty)\rightarrow C_{f}^{\ast}(X\times Y, A)$ satisfying
	$$\lim_{t\rightarrow \infty}\prop_X\left(u\left(t \right)\right)=0.$$ 
\end{definition}

\begin{remark}
	Similar definitions as above are also appeared in \cite{Deng-Guo-2024}\cite{Fu-Wang-Yu-2020}. 
\end{remark}

Let $X$, $X'$ and $Y$ be three proper metric spaces. For a uniformly continuous coarse map $g: X\rightarrow X'$, then $g\times Id: X\times Y \rightarrow X'\times Y$ is also a uniformly continuous coarse map. Thus by Lemma \ref{continuous-covering-isometry}, there exists a $\ast$-homomorphism $Ad_{g\times Id}: C^{\ast}_{L, f}(X\times Y, A)\rightarrow C^{\ast}_{L, f}(X'\times Y, A)$ defined by $V_{g\times Id}=V_g\otimes I_{\ell^2(Z_Y)}$. Moreover, for $u\in C^{\ast}_{L,X,f}(X\times Y, A)$, we have that
$$\prop_{X'}(V_{g\times Id}(t)u(t)V^{\ast}_{g\times Id}(t))\leq  w_g(\prop_{X}(u(t)))+2\delta_k,$$
for any $t\in [k,k+1]$, where $\lim_{k\rightarrow \infty}\delta_k=0$. Thus $Ad_{g\times Id}$ can be extended to a $\ast$-homomorphism from $C^{\ast}_{L,X,f}(X\times Y, A)$ to $C^{\ast}_{L,X',f}(X'\times Y, A)$. 

By the similar proofs of Lemma \ref{Lip-homotopy-invaraint} and Lemma \ref{MV-fil-localg}, we also have the following strongly Lipschitz homotopy invariance and Mayer-Vietoris six-term exact sequence for the $K$-theory of localization algebra along $X$ with filtered coefficients.

\begin{lemma}\label{Lem-Liphomo-localg-along-X}
	Let $X$, $X'$ and $Y$ be three proper metric spaces. If $X$ is strongly Lipschitz homotopy equivalent to $X'$, then $K_{\ast}(C^{\ast}_{L,X, f}(X\times Y, A))$ is naturally isomorphic to $K_{\ast}(C^{\ast}_{L,{X'}, f}(X'\times Y, A))$.
\end{lemma}

\begin{lemma}\label{Lem-MV-localg-along-X}
	Let $X=X_1\cup X_2$ be a uniformly excisive cover. If there exists $\varepsilon_0>0$ such that the $\varepsilon_0$-neighborhoods of $X_1, X_2$ and $X_1\cap X_2$ are strongly Lipschitz homotopy equivalent to $X_1, X_2$ and $X_1\cap X_2$, respectively. Then there is a Mayer-Vietoris six-term exact sequence
	$$\scriptsize\xymatrix{
		K_0(C^{\ast}_{L,X_{1,2}, f}(X_{1,2}\times Y, A)) \ar[r] & 
		\bigoplus_{i=1,2}K_0(C^{\ast}_{L,{X_i}, f}(X_i\times Y, A))
		\ar[r] & 
		K_0(C^{\ast}_{L,{X}, f}(X\times Y, A))\ar[d] \\
		K_1(C^{\ast}_{L,{X}, f}(X\times Y, A))\ar[u] &
		\bigoplus_{i=1,2}K_1(C^{\ast}_{L,{X_i}, f}(X_i\times Y, A))
		\ar[l] &
		K_1(C^{\ast}_{L,X_{1,2}, f}(X_{1,2}\times Y, A)),\ar[l] 
	}$$
	where $X_{1,2}=X_1\cap X_2$.
\end{lemma}

Now we introduce the \textit{coarse Baum-Connes conjecture along $X$ with filtered coefficients} for $X\times Y$.

\begin{conjecture}\label{Conj-CBC-along-X}
	Let $X$, $Y$ be two proper metric spaces and $N_X$ be a locally finite net in $X$. Then the evaluation at zero map $e$ induces the following isomorphism
	$$e_{\ast}: \lim_{k\rightarrow \infty} K_{\ast}(C^{\ast}_{L, P_k(N_X), f}(P_k(N_X)\times Y, A))\rightarrow \lim_{k\rightarrow \infty} K_{\ast}( C_{f}^{\ast}(P_k(N_X) \times Y, A)).$$
\end{conjecture}

\begin{lemma}\label{Lem-CBCalong-coarse-equi}
	If two proper metric spaces $Y$ and $Y'$ are coarsely equivalent, then the coarse Baum-Connes conjecture along $X$ with filtered coefficients holds for $X\times Y$ if and only if it holds for $X\times Y'$. 
\end{lemma}
\begin{proof}
	Let $g: Y\rightarrow Y'$ be a coarsely equivalent map. Then $Id\times g: P_{k}(N_X)\times Y \rightarrow P_{k}(N_X)\times Y'$ is also coarsely equivalent map, which is uniformly continuous on $P_{k}(N_X)$. Then by Corollary \ref{coarse-equi-isomor}, we can show that $K_{\ast}(C^{\ast}_{L, P_k(N_X), f}(P_k(N_X)\times Y, A))$ and $K_{\ast}(C^{\ast}_f(P_k(N_X)\times Y, A))$ are naturally isomorphic to $K_{\ast}(C^{\ast}_{L, P_k(N_X), f}(P_k(N_X)\times Y', A))$ and $K_{\ast}(C^{\ast}_f(P_k(N_X)\times Y', A))$, respectively, for any $k\in \N$. Thus we completed the proof. 
\end{proof}

Recall from Lemma \ref{filtration-Roe-algebra} that $C^{\ast}_f(Y, A)$ is a filtered $C^{\ast}$-algebra equipped with the filtration induced by the propagation on $Y$ and the filtration on $A$. Then we have the following relation between the coarse Baum-Connes conjecture along $X$ with filtered coefficients and the coarse Baum-Connes conjecture with filtered coefficients. 

\begin{proposition}\label{Prop-CBC-X-filcoeff}
	Let $X$, $Y$ be two proper metric spaces and $A$ be a filtered $C^{\ast}$-algebra acting on a Hilbert space $H_A$. If the coarse Baum-Connes conjecture with filtered coefficients in $C_{f}^{\ast}(Y, A)$ holds for $X$, then the coarse Baum-Connes conjecture along $X$ with filtered coefficients in $A$ holds for $X\times Y$. 
\end{proposition}
\begin{proof} 
	Let $H$ be a separable Hilbert space and $U:H\rightarrow H\otimes H$ be a unitary. Fixed a unit vector $h_0\in H$ and let $P: H\rightarrow H, h\mapsto \langle h, h_0 \rangle h_0$. Then $P$ is a rank-one projection. Recall that the algebra $C_{f}^{\ast}(X\times Y, A)$ acts on $\ell^2(Z_X)\otimes \ell^2(Z_Y) \otimes H\otimes H_A$ and the algebra $C^{\ast}_{f}(X, C^{\ast}(Y, A))$ acts on $\ell^2(Z_X)\otimes H \otimes \ell^2(Z_Y) \otimes H \otimes H_A$, where $Z_X$ and $Z_Y$ are two countable dense subsets in $X$ and $Y$, respectively. Define two $\ast$-homomorphisms:
	{\small\begin{equation*}
			\begin{split}
				&\phi: C^{\ast}_f(X, C^{\ast}(Y, A)) \rightarrow C_{f}^{\ast}(X\times Y, A), T_{(x,y),(x',y')} \mapsto (U^{\ast}\otimes I_{H_A})T_{(x, y),(x', y')}(U\otimes I_{H_A});\\
				& \psi: C_{f}^{\ast}(X\times Y, A)\rightarrow C^{\ast}_f(X, C^{\ast}(Y, A)), S_{(x,y),(x',y')}\mapsto S_{(x,y),(x'y')}\otimes P.
			\end{split}
	\end{equation*}}
	Then the homomorphism $\phi\circ\psi: C_{f}^{\ast}(X\times Y, A)\rightarrow C_{f}^{\ast}(X\times Y, A)$ satisfies that $$\phi\circ\psi(S)_{(x,y),(x',y')}=(U^{\ast}V\otimes I_{H_A})S_{(x,y),(x',y')}(V^{\ast}U\otimes I_{H_A}),$$ 
	where $V: H\rightarrow H\otimes H$ maps $h$ to $h\otimes h_0$. Since the operator $I_{\ell^2(Z_X)\otimes \ell^2(Z_Y)}\otimes U^{\ast}V\otimes I_{H_A}$ is an isometry in the multiplier algebra of $C^{\ast}_f(X\times Y, A)$. Thus, by Lemma \ref{isometryequivalent}, $\phi_{\ast}\circ\psi_{\ast}$ is the identity homomorphism on $K_{\ast}(C_{f}^{\ast}(X\times Y, A))$. Moreover, we can similarly define two homomorphisms between $C^{\ast}_{L,f}(X, C^{\ast}(Y, A))$ and $C_{L,X,f}^{\ast}(X\times Y, A)$, and obtain the similar result as above. Therefore, we have the following commutative diagram:
	$$\xymatrix{
		\lim_{k\rightarrow \infty}K_{\ast}(C^{\ast}_{L, P_k(N_X), f}(P_k(N_X)\times Y, A)) \ar[d]_{\psi_{\ast}} \ar[r]^{e_{\ast}} & 
		\lim_{k\rightarrow \infty}K_{\ast}(C_{f}^{\ast}(P_k(N_X)\times Y, A)) \ar[d]^{\psi_{\ast}} & 
		\\
		\lim_{k\rightarrow \infty}K_{\ast}(C^{\ast}_{L, f}(P_k(N_X), C^{\ast}(Y, A)))\ar[d]_{\phi_{\ast}} \ar[r]^{e_{\ast}} &
		\lim_{k\rightarrow \infty}K_{\ast}(C^{\ast}_{f}(P_k(N_X), C^{\ast}(Y, A)))\ar[d]^{\phi_{\ast}}
		\\
		\lim_{k\rightarrow \infty}K_{\ast}(C^{\ast}_{L, P_k(N_X), f}(P_k(N_X)\times Y, A)) \ar[r]^{e_{\ast}} & 
		\lim_{k\rightarrow \infty}K_{\ast}(C_{f}^{\ast}(P_k(N_X)\times Y, A)), 
	}$$
	where $N_X$ is a locally finite net in $X$. Since the compositions of two vertical homomorphisms are all the identity homomorphism. Thus if the middle horizontal map is an isomorphism, then the top horizontal map is injective and the bottom horizontal map is surjective. This implies the coarse Baum-Connes conjecture along $X$ with filtered coefficients in $A$ holds for $X\times Y$.
\end{proof}

\begin{remark}
	For two proper metric spaces $X$ and $Y$, the algebra $C^{\ast}_f(X\times Y, A)$ is generally not isomorphic to the algebra $C^{\ast}_f(X, C^{\ast}(Y, A))$. For example, If both $X$ and $Y$ are the separated coarse union of a sequence of single points and $A=\C$, then $C^{\ast}_f(X\times Y, A)=\ell^{\infty}(X, \ell^{\infty}(Y, \K(H)))$ and $C^{\ast}_f(X, C^{\ast}(Y, A))=\ell^{\infty}(X, \K(H)\otimes \ell^{\infty}(Y, \K(H)))$ that are not isomorphic.   
\end{remark}

Finally, we simplify the consistent coarse Baum-Connes conjecture with filtered coefficients in this subsection.

\begin{lemma}\label{uniCBC-CBCalongX}
	Let $X$ be a proper metric space, $\sqcup_{\N} X$ be the separated coarse union of $X$. Then $X$ satisfies the consistent coarse Baum-Connes conjecture with filtered coefficients in $A$ if and only if $X\times \square$ satisfies the coarse Baum-Connes conjecture along $X$ with filtered coefficients in $A$, where $\square$ is the separated coarse union of a sequence of single points.
\end{lemma}
\begin{proof}
	Firstly, we have that $\sqcup_{\N} X=X\times \square$. Thus 
	$$C_{f}^{\ast}(\sqcup_{\N}X, A)=C_{f}^{\ast}(X\times \square, A);\:\:C_{L,f}^{\ast}(\sqcup_{\N}X, A)=C_{L,f}^{\ast}(X\times \square, A).$$
	Secondly, for an operator $T\in C_{f}^{\ast}(X\times \square, A)$, we have that $\prop_{X}(T)=\prop(T)$. Thus 
	$$C^{\ast}_{L,f}(X\times \square, A)=C^{\ast}_{L, X, f}(X\times \square, A).$$
	Finally, for a locally finite net $N_X$ in $X$, we have that $P_k(\sqcup_{\N} X)=\sqcup_{\N} P_k(X)$ for any non-negative integer $k$. Thus we have
	\begin{equation*}
		\begin{split}
			& C^{\ast}_{L, f}(P_{k}(\sqcup_{\N} X), A)=C^{\ast}_{L, P_{k}(X), f}(P_{k}(X)\times \square, A);\\
			& C_{f}^{\ast}(P_{k}(\sqcup_{\N} X), A)=C_{f}^{\ast}(P_{k}(X)\times \square, A),
		\end{split}
	\end{equation*}
	which completes the proof.
\end{proof}

\begin{definition}\cite[Definition 5.14]{OyonoYu2019} \label{Def-uniform-product}
	Let $\{A^{i}\}_{i\in \mathcal{I}}$ be a family of filtered $C^{\ast}$-algebra. The \textit{uniform product} of $\{A^{i}\}_{i\in \mathcal{I}}$, denoted by $\prod^{u}_{i\in \mathcal{I}} A^{i}$, is defined to be the norm closure of the set $\bigcup_{r>0}(\prod_{i\in \mathcal{I}} A^{i}_r)$ in the direct product $\prod_{i\in \mathcal{I}} A^{i}$. 
\end{definition}

\begin{remark}
	Let $\Gamma$ be a finitely generated group equipped with the word length metric. Then the Roe algebra $C^{\ast}(\Gamma)$ is $\ast$-isomorphic to the reduced crossed product $\ell^{\infty}(\Gamma, \K(H)) \rtimes_{r} \Gamma$ (please see \cite[Proposition 5.1.3]{Brown-Ozawa-book} for the proof). By a similar argument, the Roe algebra $C^{\ast}_{f}(\Gamma, A)$ with filtered coefficients in $A$ is $\ast$-isomorphic to the reduced crossed product $(\prod^{u}_{\Gamma} (A\otimes \K(H))) \rtimes_{r} \Gamma$, where $A\otimes \K(H)$ is a filtered $C^{\ast}$-algebra equipped with the filtration $(A_r\otimes \K(H))_{r>0}$.
\end{remark}

Let $A$ be a filtered $C^{\ast}$-algebra. Then, the uniform product $\prod^{u}_{\N} (A\otimes \K(H))$ is a filtered $C^{\ast}$-algebra equipped with the filtration $(\prod_{\N} (A_r\otimes \K(H)))_{r>0}$. And we have that 
$$C^{\ast}_{f}(\square, A)=\prod^{u}_{\N} (A\otimes \K(H)),$$
where $\square$ is the separated coarse union of a sequence of single points.

Thus, combing Lemma \ref{uniCBC-CBCalongX} with Proposition \ref{Prop-CBC-X-filcoeff}, we obtain the following corollary.

\begin{corollary}\label{Cor-CBCFC-UCBC}
	Let $X$ be a proper metric space. If the coarse Baum-Connes conjecture with filtered coefficients in $\prod^{u}_{\N} (A\otimes \K(H))$ holds for $X$, then the consistent coarse Baum-Connes conjecture with filtered coefficients in $A$ holds for $X$. 
\end{corollary}

\begin{remark}
	In \cite{Zhang-quan-CBC}, the authors introduced the quantitative coarse Baum-Connes conjecture with coefficients and proved that this conjecture is equivalent to the consistent coarse Baum-Connes conjecture with coefficients. Thus by the above corollary, the coarse Baum-Connes conjecture with filtered coefficients implies that conjecture.
\end{remark}

\subsection{Main theorems}

\begin{definition}\label{Def-bounded-geometry}
	A locally finite metric space $N$ is called to have \textit{bounded geometry}, if $\sup_{x\in N} \# B(x, R)$ is finite for each $R>0$, where $\# B(x, R)$ is the cardinality of the $R$-ball centered at $x$. A proper metric space $X$ is called to have \textit{bounded geometry}, if there exists a net $N_X$ with bounded geometry.
\end{definition}

Firstly, let us discuss the Rips complexes of product metric spaces.

\begin{lemma}\label{lefthand-CBC-product}
	Let $N_1$ and $N_2$ be two locally finite metric spaces with bounded geometry, then $P_k(N_1\times N_2)$ is strongly Lipschitz homotopy equivalent to and coarsely equivalent to $P_{k}(N_1)\times P_{k}(N_2)$ for any $k\in \N$.
\end{lemma}
\begin{proof}
	Since $N_1$ and $N_2$ have bounded geometry, thus the dimensions of simplicial complexes $P_{k}(N_1)$ and $P_{k}(N_2)$ are less than some positive integer $m_k$.
	Define two maps:
	$$\rho: P_{k}\left(N_1\times N_2\right)\rightarrow P_{k}\left(N_1\right)\times P_{k}\left(N_2\right)$$ 
	by
	$$\sum_{i,j} t_{ij}\left(x_i, y_j\right)\mapsto \left(\sum_i \left(\sum_j t_{ij}\right)x_i, \sum_j \left(\sum_i t_{ij}\right)y_j\right),$$
	and 
	$$\varphi: P_{k}(N_1)\times P_{k}(N_2) \rightarrow P_{k}(N_1\times N_2)$$
	by
	$$\left(\sum_{i}t_i x_i, \sum_{j}s_j y_j\right) \mapsto \sum_{i, j}t_{i}s_{j}\left(x_i, y_j\right),$$
	where $x_i\in N_1$ and $y_j\in N_2$.
	Then we have that
	$$d_{P_k}(\rho(x),\rho(y))\leq 4m_k^2 d_{P_k}(x,y);\:\: d_{P_k}(\varphi(z), \varphi(w))\leq 2d_{P_k}(z,w)$$ 
	for any $x, y\in P_k(N_1\times N_2)$ and $z, w\in P_{k}(N_1)\times P_{k}(N_2)$ by direct computations, which imply that $\rho$ and $\varphi$ are two coarse and Lipschitz maps. Moreover, $\rho\circ\varphi=id$ and $\varphi\circ\rho$ is strongly Lipschitz homotopic to the identity map by the linear combination of them. Thus, $P_k(N_1\times N_2)$ is strongly Lipschitz homotopy equivalent to $P_{k}(N_1)\times P_{k}(N_2)$. And we have that $d(\rho\circ\varphi(z), z)=0$ and $d(\varphi\circ\rho(x), x)\leq 1$, thus $P_k(N_1\times N_2)$ is coarsely equivalent to $P_{k}(N_1)\times P_{k}(N_2)$.
\end{proof}

\begin{corollary}\label{CBC-product-sepera}
	Let $X$, $Y$ be two proper metric spaces with bounded geometry and $A$ be a filtered $C^{\ast}$-algebra. Then $X\times Y$ satisfies the coarse Baum-Connes conjecture with filtered coefficients in $A$ if and only if the following map
	$$\lim_{k\rightarrow \infty}K_{\ast}(C^{\ast}_{L,f}(P_{k}(N_X)\times P_{k}(N_Y), A)) \stackrel{e_{\ast}}{\longrightarrow} \lim_{k\rightarrow \infty}K_{\ast}(C^{\ast}_{f}(P_{k}(N_X)\times P_{k}(N_Y), A)).$$
	is an isomorphism, where $N_X$ and $N_Y$ are two nets with bounded geometry in $X$ and $Y$, respectively. 
\end{corollary}
\begin{proof}
	For any $k'\geq k$, we have the following commutative diagram:
	$$\xymatrix{
		P_k(N_X)\times P_k(N_Y) \ar[r]^{i_{kk'}}\ar[d]_{\varphi} & P_{k'}(N_X)\times P_{k'}(N_Y) \ar[d]^{\varphi} 
		\\
		P_k(N_X\times N_Y) \ar[r]_{i_{kk'}} & P_{k'}(N_X\times N_Y),
	}$$
	where $\varphi$ is defined in the proof of the above lemma.
	
	Then we further have the following commutative diagram:
	{\small$$\xymatrix{
			\lim_{k\rightarrow \infty}K_{\ast}(C^{\ast}_{L,f}(P_{k}(N_X)\times P_{k}(N_Y), A)) \ar[r]^{e_{\ast}}\ar[d]_{Ad_{\varphi,\ast}} & \lim_{k\rightarrow \infty}K_{\ast}(C_f^{\ast}(P_{k}(N_X)\times P_{k}(N_Y), A)) \ar[d]^{ad_{\varphi,\ast}} 
			\\
			\lim_{k\rightarrow \infty}K_{\ast}(C^{\ast}_{L,f}(P_{k}(N_X\times N_Y), A)) \ar[r]_{e_{\ast}} & \lim_{k\rightarrow \infty}K_{\ast}(C_f^{\ast}(P_{k}(N_X\times N_Y), A)).
		}$$}
	Thus, the proof is completed by combining Corollary \ref{coarse-equi-isomor}, Lemma \ref{Lip-homotopy-invaraint} and the above lemma.
\end{proof}

Secondly, we give the following lemma which plays a crucial role in the proof of main theorems. 

Let $X$, $Y$ be two proper metric spaces and $A$ be a filtered $C^{\ast}$-algebra. Then we have the following inclusion map:
$$\tau: C^{\ast}_{L,f}(X\times Y, A) \rightarrow C^{\ast}_{L,X,f}(X\times Y, A).$$

\begin{lemma}\label{localg-fil-localg}
	Assume $X$ has bounded geometry. Let $N_{X}$ be a net with bounded geometry in $X$ and $N_{Y}$ be a locally finite net in $Y$. If $Y$ satisfies the consistent coarse Baum-Connes conjecture with filtered coefficients in $A$, then for each $k>0$
	{\footnotesize$$\tau_{\ast}: \lim_{d\rightarrow \infty} K_{\ast}\left(C^{\ast}_{L,f}\left(P_{k}\left(N_X \right)\times P_{d}\left(N_Y \right), A \right)\right) \rightarrow \lim_{d\rightarrow \infty} K_{\ast}(C^{\ast}_{L,P_{k}(N_X),f}(P_{k}(N_X)\times P_{d}(N_Y), A))$$}
	is an isomorphism.  
\end{lemma}
\begin{proof}
	Since $N_X$ has bounded geometry, thus $P_k(N_X)$ is a finite-dimensional simplicial complex for each $k$. Let $P_k(N_X)^{(n)}$ be its $n$-dimensional skeleton. We will prove the lemma for $P_k(N_X)^{(n)}$ by induction on $n$.
	
	When $n=0$, we have the following commutative diagram:
	$$\xymatrix{
		\lim_{d\rightarrow \infty} K_{\ast}(C^{\ast}_{L,f}(\sqcup_{\N}P_{d}(N_{Y}), A)) \ar[r]^{\tau_{\ast}\:\:\:\:\:\:\:\:\:\:\:\:\:\:\:} \ar[dr]_{e_{\ast}} & \lim_{d\rightarrow \infty} K_{\ast}(C_{ub}([0,\infty), C_{f}^{\ast}(\sqcup_{\N}P_d(N_{Y}), A))) \ar[d]^{e_{\ast}} 
		\\
		&  \lim_{d\rightarrow \infty} K_{\ast}(C_{f}^{\ast}(\sqcup_{\N}P_{d}(N_{Y}), A)),
	}$$
	where $\sqcup_{\N}P_{d}(N_{Y})$ is the separated coarse union.
	Then by Lemma \ref{continuous-flow-K} and Lemma \ref{Roe-quasi-stable}, we obtain that the right-hand side vertical map $e_{\ast}$ is an isomorphism. Thus by the assumption, $\tau_{\ast}$ is an isomorphism for the case of $n=0$.
	
	Assume by induction that the lemma holds for $n=l-1$. Next we will prove the lemma also holds for $n=l$. For each simplex $\triangle$ of dimension $l$ in $P_{k}(N_X)$, let $c(\triangle)$ be the center of $\triangle$. Then define
	$$\triangle_1=\{x\in \triangle: d(x, c(\triangle))\leq 1/10\},\:\: \triangle_1=\{x\in \triangle: d(x, c(\triangle))\geq 1/10\}.$$
	And let 
	$$X_1=\cup\{\triangle_1: \text{$\triangle$ is a simplex of dimension $l$}\};$$
	$$X_2=\cup\{\triangle_2: \text{$\triangle$ is a simplex of dimension $l$}\}.$$
	Then $P_{k}(N_X)^{(l)}=X_1\cup X_2$ and $X_1 \cap X_2$ is strongly Lipschitz homotopy equivalent to the disjoint union of the boundaries of all $l$-dimensional simplexes in $P_{k}(N_X)$. Besides, $X_1$ and $X_2$ are strongly Lipschitz homotopy equivalent to $\{c(\triangle): \text{$\triangle$ is a simplex of dimension $l$}\}$ and $P_{k}(N_X)^{(l-1)}$, respectively. Thus by Lemma \ref{Lip-homotopy-invaraint}, Lemma \ref{MV-fil-localg}, Lemma \ref{Lem-Liphomo-localg-along-X}, Lemma \ref{Lem-MV-localg-along-X} and the five lemma, we completed the proof.
\end{proof}	

Now, let us state and prove our main theorems. 

\begin{theorem}\label{main-THM}
	Let $X$, $Y$ be two proper metric spaces with bounded geometry and $A$ be a filtered $C^{\ast}$-algebra. If the following two conditions are satisfied: 
	\begin{enumerate}
		\item \label{cond1} the coarse Baum-Connes conjecture along $X$ with filtered coefficients in $A$ holds for $X\times Y$;  
		\item \label{cond2} the consistent coarse Baum-Connes conjecture with filtered coefficients in $A$ holds for $Y$. 
	\end{enumerate}
	Then the coarse Baum-Connes conjecture with filtered coefficients in $A$ holds for $X\times Y$. 
\end{theorem}
\begin{proof}
	Let $N_X$, $N_Y$ be two nets with bounded geometry in $X$ and $Y$, respectively. Then combining Condition (\ref{cond1}) with Lemma \ref{coarse-equi-Rips} and Lemma \ref{Lem-CBCalong-coarse-equi}, we have that $X\times P_{d}(N_Y)$ satisfies the coarse Baum-Connes conjecture along $X$ with filtered coefficients in $A$ for any $d\in \N$, namely, the following map is an isomorphism
	$$\lim_{k\rightarrow \infty}K_{\ast}(C^{\ast}_{L, P_k(N_X), f}(P_{k}(N_X)\times P_{d}(N_Y), A)) \stackrel{e_{\ast}}{\longrightarrow} \lim_{k\rightarrow \infty}K_{\ast}(C_{f}^{\ast}(P_{k}(N_X)\times P_{d}(N_Y), A)).$$
	On the other hand, by Condition (\ref{cond2}) and Lemma \ref{localg-fil-localg}, we have the following isomorphism for any $k\in \N$:
	$$\lim_{d\rightarrow \infty}K_{\ast}(C^{\ast}_{L,f}(P_{k}(N_X)\times P_{d}(N_Y), A))\stackrel{\tau_{\ast}}{\longrightarrow} \lim_{d\rightarrow \infty}K_{\ast}(C^{\ast}_{L, P_k(N_X), f}(P_{k}(N_X)\times P_{d}(N_Y), A)).$$
	Combining the above two isomorphisms, we obtain an isomorphism:
	$$\lim_{k, d\rightarrow \infty}K_{\ast}(C^{\ast}_{L, f}(P_{k}(N_X)\times P_{d}(N_Y), A)) \stackrel{e_{\ast}}{\longrightarrow} \lim_{k, d\rightarrow \infty}K_{\ast}(C_{f}^{\ast}(P_{k}(N_X)\times P_{d}(N_Y), A)),$$
	which implies that $X\times Y$ satisfies the coarse Baum-Connes conjecture with filtered coefficients in $A$ by Corollary \ref{CBC-product-sepera}.
\end{proof}

Let $A$ be a filtered $C^{\ast}$-algebra. Recall that the Roe algebra $C^{\ast}_f(Y, A)$ is a filtered $C^{\ast}$-algebra equipped with the filtration $(C^{\ast}_f(Y, A))_{r>0}$ defined in Lemma \ref{filtration-Roe-algebra} and the uniform product $\prod^{u}_{\N}(A\otimes \K(H))$ is also a filtered $C^{\ast}$-algebra equipped with the filtration $(\prod_{\N} (A_r\otimes \K(H)))_{r>0}$. Combining Theorem \ref{main-THM} with Proposition \ref{Prop-CBC-X-filcoeff} and Corollary \ref{Cor-CBCFC-UCBC}, we have the following theorem which implies that the coarse Baum-Connes conjecture with filtered coefficients is closed under products.

\begin{theorem}\label{main-corollary}
	Let $X$, $Y$ be two proper metric spaces with bounded geometry and $A$ be a filtered $C^{\ast}$-algebra. If the following two conditions are satisfied:
	\begin{enumerate}
		\item $X$ satisfies the coarse Baum-Connes conjecture with filtered coefficients in the Roe algebra $C^{\ast}_{f}(Y, A)$;  
		\item $Y$ satisfies the coarse Baum-Connes conjecture with filtered coefficients in the uniform product $\prod^{u}_{\N}(A\otimes \K(H))$.
	\end{enumerate}
	Then $X\times Y$ satisfies the coarse Baum-Connes conjecture with filtered coefficients in $A$. 
	
	In particular, if $X$ and $Y$ satisfy the coarse Baum-Connes conjecture with filtered coefficients, then $X\times Y$ satisfies the coarse Baum-Connes conjecture with filtered coefficients.
\end{theorem}

\begin{remark}
	For two discrete groups $\Gamma_1$ and $\Gamma_2$, Oyono-Oyono proved that $\Gamma_1\times \Gamma_2$ satisfies the Baum-Connes conjecture with coefficients if and only if $\Gamma_1$ and $\Gamma_2$ satisfy the Baum-Connes conjecture with coefficients in \cite[Corollary 7.12]{Oyono-BC-extensions}. Thus, we can see Theorem \ref{main-corollary} as a coarse analogue of Oyono-Oyono's result. Besides, for any proper metric space $X$, the product $X\times [0, \infty)$ always satisfies the coarse Baum-Connes conjecture with filtered coefficients by an argument of the Eilenberg swindle (see \cite[Proposition 7.5.2]{WillettYu-Book}). However, we know that the coarse Baum-Connes conjecture with filtered coefficients does not hold for some expanders. Thus, Theorem \ref{main-corollary} is not a necessary and sufficient condition in general.  
\end{remark}

When $A=\C$, we have that $\prod^{u}_{\N}(A\otimes \K(H))=\prod_{\N} \K(H)$ equipped with the trivial filtration, namely, $(\prod_{\N} \K(H))_r=\prod_{\N} \K(H)$ for each $r>0$. Thus, we obtain the following corollary which enlarges the class of product metric spaces that satisfy the coarse Baum-Connes conjecture.

\begin{corollary}\label{corollay-CBC}
	Let $X$ and $Y$ be two proper metric spaces with bounded geometry. If the followings are satisfied: 
	\begin{enumerate}
		\item $X$ satisfies the coarse Baum-Connes conjecture with filtered coefficients in the Roe algebra $C^{\ast}(Y)$;
		\item $Y$ satisfies the coarse Baum-Connes conjecture with coefficients in $\prod_{\N} \K(H)$.
	\end{enumerate}
	Then $X\times Y$ satisfies the coarse Baum-Connes conjecture. 
	
	In particular, if $X$ and $Y$ satisfy the coarse Baum-Connes conjecture with filtered coefficients, then $X\times Y$ satisfies the coarse Baum-Connes conjecture.
\end{corollary}

\begin{remark}
	For any $p\in [1,\infty)$, there is a $\ell^p$ coarse Baum-Connes conjecture for computing the $K$-theory of $\ell^p$ Roe algebras (see \cite{Zhang-Zhou-lpCBC}). Actually, Theorem \ref{main-corollary} and Corollary \ref{corollay-CBC} also hold for the $\ell^p$ coarse Baum-Connes conjecture by the similar proofs.
\end{remark}

\section{Applications}\label{Section-ex}

In this section, we illustrate that three ways mentioned in the introduction to the coarse Baum-Connes conjecture also work for the conjecture with filtered coefficients. The reason is that all morphisms appeared in the proofs of those three ways are the identity map on filtered coefficients. The results in this section recover Fukaya and Oguni's result in \cite[Theorem 1.1]{Fukaya-Oguni-2015} and also provide some new examples of product metric spaces satisfying the coarse Baum-Connes conjecture. 

\subsection{Metric spaces with finite asymptotic dimension}
\begin{definition}\label{Def-asy}
	Call a metric space $X$ has \textit{finite asymptotic dimension}, if there is a non-negative integer $m<\infty$ satisfying that for any $r>0$, there exists a uniformly bounded cover $\{U_i\}_{i\in I}$ of $X$ such that no ball with radius $r$ in $X$ intersects more than $m+1$ members of $\{U_i\}_{i\in I}$.  
\end{definition}

\begin{example}
	Every polycyclic group (which is solvable) equipped with a proper metric has finite asymptotic dimension (cf. \cite[Chapter 2]{Nowak-Yu-Book}).
\end{example}

In \cite{Yu1998}, Yu introduced the quantitative $K$-theory to prove that the coarse Baum-Connes conjecture holds for metric spaces with finite asymptotic dimension. 

Let $X$ be a proper metric space and $A$ be a filtered $C^{\ast}$-algebra. Then the Roe algebra $C_f^{\ast}(X, A)$ is also a filtered $C^{\ast}$-algebra equipped with the filtration defined by
$$C_f^{\ast}(X, A)_s=\{T\in C_f^{\ast}(X, A): \prop(T)\leq s\},$$
for each $s\geq 0$ (please compare it with Lemma \ref{filtration-Roe-algebra}). Then by using the quantitative $K$-theory to $C_f^{\ast}(X, A)$ with the above filtration, we can obtain the following result. 

\begin{proposition}\label{propo-CBCFC-asy}
	Let $X$ be a proper metric space with finite asymptotic dimension and $A$ be a filtered $C^{\ast}$-algebra, then the coarse Baum-Connes conjecture with filtered coefficients in $A$ holds for $X$.
\end{proposition}

\subsection{Metric spaces with coarse embeddability into Hilbert space and CE-by-CE extensions}

\begin{definition}\label{Def-coarse-emd}
	A proper metric space $X$ is called to be \textit{coarsely embedded into Hilbert space $H$}, if there exists a map $f: X\rightarrow H$ and two non-decreasing functions $\rho_{\pm}:[0,\infty)\rightarrow [0,\infty)$ with $\lim_{t\rightarrow \infty} \rho_{-}(t)=\infty$ such that 
	$$\rho_{-}(d(x,x'))\leq \parallel f(x)-f(x') \parallel \leq \rho_{+}(d(x,x')),$$
	for all $x,x'\in X$.
\end{definition}

\begin{example}
	The class of metric spaces which admit a coarse embedding into Hilbert space contains hyperbolic spaces, mapping class groups and linear groups equipped with the word length metric (cf. \cite{Nowak-Yu-Book}).
\end{example}

In \cite{Yu2000}, Yu used the Dirac-dual-Dirac method to prove that the coarse Baum-Connes conjecture holds for any metric space which admits a coarse embedding into Hilbert space (see \cite[Chapter 12]{WillettYu-Book} for a simplified proof). We can take tensor products of the Dirac map and the dual-Dirac map with the identity map on filtered coefficients in the Dirac-dual-Dirac method, then we obtain the following result. 

\begin{proposition}\label{propo-CBC-coarse-embed}
	Let $X$ be a proper metric space with bounded geometry and $A$ be a filtered $C^{\ast}$-algebra. If $X$ can be coarsely embedded into Hilbert space, then the coarse Baum-Connes conjecture with filtered coefficients in $A$ holds for $X$.
\end{proposition}

The following notion was introduced by Deng, Wang and Yu in \cite{Deng-Wang-Yu-2023} in order to consider the coarse Baum-Connes conjecture for some relative expanders. Recall that the \textit{coarse union} of a sequence of metric spaces $(X_i)_{i\in \N}$ is a disjoint union $\sqcup_{i\in \N} X_i$ equipped with a metric $d$ that restricts to the original metric on each $X_i$ and satisfies $d(X_n, (\sqcup_{\N}X_i)\setminus X_n)\rightarrow \infty$ as $n\rightarrow \infty$.

\begin{definition}\label{Def-CE-BY-CE}
	A sequence of group extensions $(1\rightarrow N_m\rightarrow \Gamma_m\rightarrow Q_m\rightarrow 1)_{m\in \N}$ is called to be a \textit{CE-by-CE extension}, if both the coarse union of $(N_m)_{m\in \N}$ and the coarse union of $(Q_m)_{m\in \N}$ admit a coarse embedding into Hilbert space. 
\end{definition}

\begin{example}
	Some relative expanders provide examples of CE-by-CE extensions, such as the box spaces of groups $\Z^2 \rtimes_{Q} \mathbb{F}_3$ and $\Z \wr_{Q} \mathbb{F}_3$, where $Q$ is the kernel of the quotient map from $SL(2, \Z)$ to $SL(2, \Z_2)$ (cf. \cite{Deng-Wang-Yu-2023}).
\end{example}

Using the same arguments (relies on the Dirac-dual-Dirac method) in \cite{Deng-Wang-Yu-2023}, we can obtain the following proposition.

\begin{proposition}\label{propo-CBC-cebyce}
	Let $(1\rightarrow N_m\rightarrow \Gamma_m\rightarrow Q_m\rightarrow 1)_{m\in \N}$ be a CE-by-CE extension with all coarse unions of $(N_m)_{m\in \N}$, $(\Gamma_m)_{m\in \N}$ and $(Q_m)_{m\in \N}$ have bounded geometry. Then the coarse Baum-Connes conjecture with filtered coefficients in any filtered $C^{\ast}$-algebra holds for the coarse union of $(\Gamma_m)_{m\in \N}$.  
\end{proposition}

\begin{remark}
	Recently, Deng and Guo in \cite{Deng-Guo-2024} introduced a twisted coarse Baum-Connes conjecture and they showed that the coarse Baum-Connes conjecture holds for the coarse union of $(\Gamma_m)_{m\in \N}$ if coarse unions of $(N_m)_{m\in \N}$ and $(Q_m)_{m\in \N}$ satisfy the twisted coarse Baum-Connes conjecture with coefficients which generalizes the main result in \cite{Deng-Wang-Yu-2023}. Actually, we can obtain the similar result for the coarse Baum-Connes conjecture with filtered coefficients.
\end{remark}

\subsection{Open cones, coarse homotopy and relatively hyperbolic groups}

\begin{definition}\label{Def-opencone}
	Let $(M, d_M)$ be a compact metric space with diameter at most two. The \textit{open cone} over $M$, denoted by $\mathcal{O}M$, is defined to be the quotient space $(\R_{\geq 0}\times M)/(\{0\}\times M)$ with the following metric 
	$$d((t,x),(s,y))=\lvert t-s \rvert+\min\{t,s\}d_{M}(x,y),$$
	for any $t,s\in \R_{+}$ and $x,y\in M$.
\end{definition}

\begin{definition}\cite[Definition 2.1]{Fukaya-Oguni-2020} \label{DefCoarHom}
	Let $f,g: X \rightarrow Y$ be two coarse maps between proper metric spaces. $f$ and $g$ are called to be \textit{coarsely homotopic}, if there exists a metric subspace $Z=\{(x,t): 0\leq t\leq t_x\}$ of $X\times \mathbb{R}$ and a coarse map $h: Z\rightarrow Y$, such that 
	\begin{itemize}
		\item the map from $X$ to $\mathbb{R}$ given by $x\mapsto t_x$ satisfies that for any $R\geq 0$, there exists $S\geq 0$ such that $|t_{x}-t_{x'}|\leq S$ for all elements $x,x'\in X$ with $d(x,x')\leq R$,
		\item $h(x,0)=f(x)$, $h(x,t_x)=g(x)$.
	\end{itemize}
	The map $f$ is called a \textit{coarse homotopy equivalence map} if there exists a coarse map $f':Y\rightarrow X$ such that $f'f$ and $ff'$ are coarsely homotopic to $id_X$ and $id_Y$, respectively. Call $X$ and $Y$ be \textit{coarsely homotopy equivalent} if there exists a coarse homotopy equivalence map from $X$ to $Y$.
\end{definition}

Every open cone satisfies the coarse Baum-Connes conjecture by an argument of the Eilenberg swindle. Its proof can be applied to the conjecture with filtered coefficients in any filtered $C^{\ast}$-algebras (cf. \cite{Zhang-opencone}, \cite{HigsonRoe-CBC}). Moreover, the conjecture with filtered coefficients is a coarse homotopy invariant (the proof relies on the Mayer-Vietoris six-term exact sequence, cf. \cite{Zhang-opencone}, \cite[Chapter 12]{HigsonRoe-Book}). Thus we have the following proposition.

\begin{proposition}\label{Thm-CBC-opencone}
	If a proper metric space $X$ is coarsely homotopy equivalent to an open cone, then the coarse Baum-Connes conjecture with filtered coefficients in any filtered $C^{\ast}$-algebra holds for $X$.
\end{proposition}

\begin{example}\label{Exam-opencones}
	In \cite{Fukaya-Oguni-2020}, Fukaya and Oguni provided a class of metric spaces which are coarsely homotopy equivalent to an open cone, including
	\begin{enumerate}
		\item all hyperbolic spaces;
		\item Busemann non-positively curved spaces, for example, CAT($0$)-spaces;
		\item systolic groups with the word length metric, especially, Artin groups of almost large type and graphical small cancellation groups;
		\item Helly groups with the word length metric, especially, weak Garside groups of finite type and FC-type Artin groups (\cite{Huang-Osajda-Helly}).
	\end{enumerate}
\end{example}

\begin{remark}
	It is unknown whether every CAT($0$)-space or Helly group admits a coarse embedding into Hilbert space.
\end{remark}

In \cite{Fukaya-Oguni-2012}, Fukaya and Oguni proved that if a group is hyperbolic relative to a finite family of infinite subgroups (cf. \cite[Definition 2.4]{Fukaya-Oguni-2012}), then this group satisfies the coarse Baum-Connes conjecture provided each subgroup belonging to the family satisfies the conjecture and admits a finite universal space for proper actions. By a similar argument, we can prove the following proposition.

\begin{proposition}\label{propo-CBC-relatively-hyperbolic}
	Let $\Gamma$ is a finitely generated group which is hyperbolic relative to a finite family of infinite subgroups $\{N_1, \cdots, N_k\}$ and $A$ be a filtered $C^{\ast}$-algebra. If every $N_i$ satisfies the coarse Baum-Connes conjecture with filtered coefficients in $A$ and admits a finite-dimensional simplicial complex which is a universal space for proper actions, then $\Gamma$ satisfies the coarse Baum-Connes conjecture with filtered coefficients in $A$.
\end{proposition}  

In \cite{Fukaya-Oguni-2015}, Fukaya and Oguni showed that the coarse Baum-Connes conjecture holds for products of hyperbolic groups, CAT($0$)-groups, polyclic groups and certain relatively hyperbolic groups by considering the coronae of products. But it is generally unknown whether the coarse Baum-Connes conjecture is true or not for products when their factors satisfy the conjecture.

To sum up this section, by Theorem \ref{main-corollary}, we can obtain that the products of metric spaces appeared in this section satisfy the coarse Baum-Connes conjecture with filtered coefficients. In particular, these products satisfy the coarse Baum-Connes conjecture. That not only recovers Fukaya and Oguni's result in \cite[Theorem 1]{Fukaya-Oguni-2015} by combining Examples \ref{Exam-opencones} with Proposition \ref{Thm-CBC-opencone} and Proposition \ref{propo-CBC-relatively-hyperbolic}, but also gives some new examples of products satisfying the coarse Baum-Connes conjecture, such as the products of CAT($0$)-groups or Helly groups with mapping class groups or linear groups equipped with the word length metric.

\bibliographystyle{plain}
\bibliography{CBCproducts}

\end{document}